\documentclass[12pt, a4paper]{amsart}
\usepackage[utf8]{inputenc}
\usepackage[greek, british]{babel}
\usepackage{microtype}
\usepackage{textcomp}
\usepackage{lmodern, fixcmex}
\usepackage[scr = dutchcal, scrscaled = 1.05]{mathalfa}
\usepackage{indentfirst}

\usepackage[
backend = biber, 
style=alphabetic,
citestyle=alphabetic,
sorting=nyvt
]{biblatex}
\AtBeginBibliography{\footnotesize}

\usepackage{graphicx, tikz, tikz-cd}
\usetikzlibrary{decorations.markings,arrows,arrows.meta, patterns, shapes}
\usepackage{amsmath, mathtools, amssymb, amsfonts, amsthm}
\usepackage[inline]{enumitem}
\usepackage{calc}
\usepackage{geometry}
\usepackage[
colorlinks=true, 
urlcolor=blue!50!black, 
linkcolor=black, 
citecolor=black, 
raiselinks=false,
pdfa
]{hyperref}
\usepackage[
english = british,
italian = quotes
]{csquotes}
\usepackage[unicode]{genstylemini}

\setlist[enumerate]{
	topsep = .5ex,
	itemsep=.5ex,
	label=\arabic*.,
	ref  =\arabic*.
}
\setlist[itemize]{
	itemsep=.5ex,
	topsep = .5ex
}

\newtheorem{theo}{Theorem}[section]
\newtheorem{lemma}[theo]{Lemma}
\newtheorem{corol}[theo]{Corollary}
\newtheorem{propos}[theo]{Proposition}
\newtheorem{fact}[theo]{Fact}

\theoremstyle{definition}
\newtheorem{defin}[theo]{Definition}

\newtheorem{question}[theo]{Question}

\theoremstyle{remark}
\newtheorem{obs}[theo]{Remark}

\usepackage{stackengine}
\newcommand{\sytranslate}[3]{\raisebox{#2}{\kern#1#3}} % translates #3 by #2 vertically and by #1 horizontally

\newcommand\syoverlap[2][c]{%
	\bgroup%
	\setstackEOL{,}% symbols to overlap separated by a comma
	\setstackgap{L}{0pt}%
	\Longstack[#1]{#2}%
	\egroup%
} %overlap #a and #b, where #2=#a,#b. #1 is an optional alignement (possible values: c, l, r ?)

\newcommand{\indep}[2][]{%
	\mathrel{
		\mathop{
			\vcenter{
				\hbox{\oalign{\noalign{\kern-.3ex}\hfil\scalebox{1}[.8]{$\vert$%
							\makebox[0pt]{\raisebox{.9ex}{\kern.3ex\scalebox{.5}{$\mathsf{#2}$}}}}%
						\hfil\cr
						\noalign{\kern-.7ex}
						$\smile$\cr\noalign{\kern-.3ex}}}
			}
		}\displaylimits_{#1}
	}
} % anchor symbol for independence, use this way \newcommand{\myind}[1][]{\indep[#1]{[superscript-symbol]}}
\newcommand{\nindep}[2][]{%
	\mathrel{
		\mathop{
			\vcenter{
				\hbox{\oalign{%
						\noalign{\kern-.3ex}\hfil\scalebox{1}[.8]{$\vert$%
							\makebox[0pt]{\raisebox{.9ex}{\kern.3ex\scalebox{.5}{$\mathsf{#2}$}}}\makebox[0pt]{\kern-.7ex$\smallsetminus$}}%
						\hfil\cr
						\noalign{\kern-.7ex}
						$\smile$\cr\noalign{\kern-.3ex}}}
			}
		}\displaylimits_{#1}
	}
} % negation of the independence symbol

\makeatletter
\newcommand{\dotminus}{\mathbin{\text{\@dotminus}}}
\newcommand{\@dotminus}{%
	\ooalign{\hidewidth\raise.9ex\hbox{.}\hidewidth\cr$\m@th-$\cr}%
}
\makeatother
\let\supp\relax
\DeclareMathOperator{\supp}{supp}

\newcommand{\lowersub}[1]{\smash{_{{\footnotesize\textstyle\mathstrut}#1}}}

\newcommand{\meet}{\mathbin{\wedge}}
\newcommand{\join}{\mathbin{\vee}}

\newcommand{\symdiff}{\mathop{\triangle}} % symmetric difference
\newcommand{\tsigma}{\textgreek{\textsigma}} % text sigma

\newcommand{\dd}{\mathop{}\mathopen{}\mathrm{d}}

\DeclareMathOperator{\acl}{acl}
\DeclareMathOperator{\dcl}{dcl}

\newcommand{\lBL}{£L_{\mathrm{BL}}}

\newcommand{\alpl}[1][p]{\ensuremath{\mathrm{AL_{\mathit{#1}}L}}}

\newcommand{\boundary}{\partial}
\newcommand{\actson}{\curvearrowright}
\newcommand{\charfun}[1]{χ\lowersub{#1}}
\DeclareMathOperator{\stone}{Stone} %stone space of
\newcommand{\seq}[1]{\harpoonhat{#1}}

\DeclareMathOperator{\tilenum}{tn}
\newcommand{\inboundary}[2]{\boundary_{#1}^{\mathrm{in}}#2}

\DeclareMathOperator{\bandsymb}{€b}
\newcommand{\band}[2][]{\bandsymb_{#1}§(#2)}

\newgeometry{left=2.5cm,right=2.5cm,top=3cm,bottom=3cm} %manual oneside
\numberwithin{equation}{section}

\newcommand{\tgfree}{T_{G}^{\mathrm{free}}}

\begin{document}
	
	\title{Amenable group actions on $L_p$ lattices}
	\author{Antonio~M.~Scielzo}
	
	\begin{abstract}
		A result by Ornstein and Weiss states that a free and measure-preserving action of an amenable group on a probability space yields a decomposition of the space in disjoint images, up to a small error, analogous to the one given by the Rokhlin lemma in the case of a single transformation. We generalise this result to non-singular actions, and use it to prove that the theory of an action by automorphisms of an amenable group on a Banach $L_p$ lattice admits a model companion, which is stable and has quantifier elimination.
	\end{abstract}

	\maketitle
	\tableofcontents

\section{Introduction}

A classical result in ergodic theory is the Rokhlin lemma, which asserts that given an aperiodic measure-preserving transformation $σ$ of a standard measure space $(X,£F,μ)$, $ε>0$, and $n>0$, there exists some $A∈ £F$ such that $A$, $σ(A)$, …, $σ^{n} A$ are pairwise disjoint and $μ(X \setdiff \bigcup_{i⩽n} σ^i A) < ε$. Here, by aperiodic we mean that sets of $m$-periodic points $§{x∈X : σ^m(x) = x}$ are of measure zero for all $m>0$.

This result may be generalised to the case where $σ$ is an aperiodic non-singular transformation of $(X,£F,μ)$, i.e., it satisfies $μ(A) = 0$ if and only if $μ(σA) = 0$. The earliest reference the author could find for this is \cite[p.~71]{halmos}.

In another direction, Ornstein and Weiss have proved a generalisation of the Rokhlin lemma in \cite{o-w} to free measure-preserving actions of an amenable group $G$ on a standard probability space $(X,£F,μ)$. Given subsets $T_1,…,T_\ell$ of $G$, we say that they \emph{tile} $G$ if they are finite and there exist $C_1,…,C_\ell ⊆G$ such that $§{T_ic : c∈ C_i, 1⩽ i⩽ \ell}$ is a partition of $G$. The notation $Ta$ is used to denote the set $§{ta : t ∈T}$, whether $a$ belongs to $G$ or to $X$. Similarly, $TA$ stands for $§{ta : t∈T, a ∈ A}$, for any $A⊆G$ or $A⊆X$.
The result by Ornstein and Weiss may then be phrased as follows. 
\begin{theo}[{\cite[Theorem~II.2.5]{o-w}}]\label{o-w}
	Let $G$ be a countable amenable group that acts freely and preserving the measure on a standard probability space $(X,£F,μ)$, and suppose that $T_1,…,T_\ell ⊆G$ tile $G$. Then there exist measurable sets $W_1,…,W_\ell∈ £F$ such that
	\begin{enumerate}
		\item for each $1⩽i⩽\ell$, if $t,s∈ T_i$ are different, then $tW_i$ and $sW_i$ are disjoint,
		\item if $1⩽ i<j ⩽ \ell$, then $T_iW_i$ and $T_jW_j$ are disjoint,
		\item $μ§(\bigcup_{1⩽i⩽\ell} T_i W_i) > 1-ε$.
	\end{enumerate}
\end{theo}
Notice that if $G = \zz$, then an interval $§{0,…,n}$ on its own tiles $G$, and freeness is the same as aperiodicity, so the theorem above becomes precisely the classical Rokhlin lemma for a probability space.

The first objective of this paper, which we will carry out in \Cref{ns actions}, is to prove a third generalisation of the Rokhlin lemma that combines the previous two. We will show that we may drop the condition of measure preservation from \Cref{o-w}, and replace it with the weaker notion of non-singularity, thus completing the following square of generalisations.
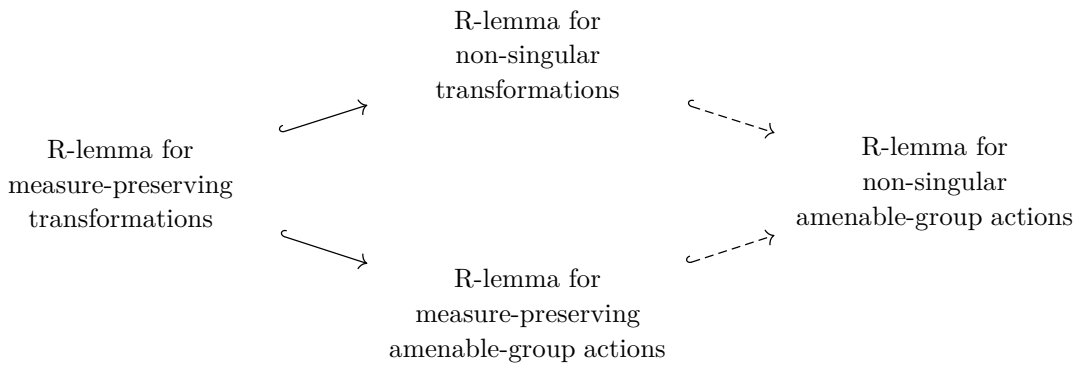
\begin{equation*}\tikzcdset{diagrams={arrows={shorten >=-.5ex,shorten <=-.5ex}}}
	\begin{tikzcd}[row sep = 2mm]
		& \parbox{4cm}{\footnotesize\centering R-lemma for \\ non-singular transformations} \arrow[dr, hookrightarrow, dashed] & \\
		\parbox{4cm}{\footnotesize\centering R-lemma for \\ measure-preserving transformations} \arrow[ur, hookrightarrow] \arrow[dr, hookrightarrow] & & \parbox{4cm}{\footnotesize\centering R-lemma for \\ non-singular amenable-group actions}\\
		& \parbox{4cm}{\footnotesize\centering R-lemma for \\ measure-preserving amenable-group actions} \arrow[ur, hookrightarrow, dashed]
	\end{tikzcd}
\end{equation*}

A result by Downarowicz, Huczek, and Zhang asserts that if $G$ is a countable amenable group, then we may find, for any finite $K⊆G$ and any $ε>0$, some subsets $T_1,…,T_\ell ⊆ G$ that tile $G$ and are \emph{$(K,ε)$-invariant}, i.e., such that the set $\boundary_KT_i = K^{-1}T_i ∩ K^{-1}(G\setdiff T_i)$ contains fewer than $ε \card{T_i}$ elements, for each $i⩽ \ell$.
If the action of $G$ is measure preserving, then the images $(\boundary_K T_i ) W_i$ of the sets $W_i$ given by \Cref{o-w} are automatically of measure smaller than $ε μ(T_iW_i)$.  A priori, one would lose this property when moving to the non-singular case. We will show however that by carefully choosing the sets $W_i$ and taking very invariant subsets $T_1,…,T_\ell ⊆ G$, we can still have the images $(\boundary_K T_i ) W_i$ of small measure with respect to $μ(T_iW_i)$.

The second objective of this paper is to apply this result to the model theory of group actions by automorphisms on Banach $L_p$ lattices. Recall that an $L_p$ lattice is a Banach space of the form $L^p(X,£F,μ)$ over a measure space $(X,£F,μ)$, endowed with the usual order, $f⩽g$ if $f(x) ⩽ g(x)$ for almost all $x∈X$, which makes it a lattice.

The link between the measure-space setting and the $L_p$-lattice one is given in \Cref{section : representation}, where we show a representation result stating that, given an $L_p$ lattice $E$ and an action $ρ$ of $G$ on it, there exists a measure space $(X,£F,μ)$ and an action $π$ of $G$ by non-singular transformations on $(X,£F,μ)$ such that $E$ is isomorphic to $L^p(X,£F,μ)$ and $ρ$ is induced by $π$.
Moreover, this isomorphism may be chosen in such a way as to make any maximal disjoint subset of positive elements of $E$ correspond to a set of characteristic functions in $L^p(X)$.
This will enable us to transfer the Rokhlin lemma for actions on probability spaces by non-singular transformations to the context of $L_p$ lattices, to obtain a local decomposition result for actions on these structures.

In \Cref{GLp lattices}, we study the theory $T_G$ of \emph{$G$-$L_p$ lattices}, that is, $L_p$ lattices acted upon by a countable group $G$. We introduce a suitable notion of “functional” freeness for $G$-$L_p$ lattices, which coincides with the one induced by the correspondence described in the previous paragraph, and which can be expressed as a set of axioms in the language of $G$-$L_p$ lattices. Expanding $T_G$ with these axioms, we obtain a new theory $\tgfree$, which we will prove to be the model companion of $T_G$, when $G$ is amenable. The decomposition result mentioned above will be the fundamental tool for this proof.

This generalises the results in \cite[Section~3]{scielzo}, where we studied the case $G = \zz$, that is, the theory of $L_p$ lattices with a distinguished automorphism. Analogously to the case of $\zz$, we will show that this new theory $\tgfree$ is stable and has quantifier elimination.

\section{Non-singular actions of amenable groups}\label{ns actions}
In this section, we generalise the decomposition result by Ornstein and Weiss to the case of non-singular actions of amenable groups on probability spaces. 

\subsection{Approximate multicovers} We shall start by introducing the notion of approximate multicover, and show that it is possible to extract an approximate cover from it with a given bound on the overlapping of its elements. 

Throughout this section, $(X,£F,μ)$ will be a probability space.

\begin{defin}
	A family $§{U_i}_{i∈I}$ of measurable subsets of $X$ is called a \emph{multicover} of $X$ with $M$ layers and error $δ$ if
	\begin{enumerate}[label=(\roman*)]
		\item \label{first multicover} $\displaystyle \sum_{i∈I} \charfun{U_i} ⩽ M$ a.e.
		\item \label{second multicover} $\displaystyle \sum_{i∈I} μ(U_i) ⩾ (1-δ) M$ .
	\end{enumerate}
\end{defin}

The following two lemmas are slightly modified versions of \cite[I.\S 2 Lemma 2]{o-w} and \cite[II.\S 2 Lemma 4]{o-w}, respectively, where we have incorporated the error $δ$.

\begin{lemma}
	Let $§{U_i}_{i∈I}$ be a multicover with error $δ$. If $V$ is a measurable subset of $X$, then there exists $i∈I$ such that 
	\begin{equation}\label{basic lemma}
	μ(U_i ∩ V) ⩽ \frac{μ(U_i)\, μ(V)}{1-δ}.
	\end{equation}
\end{lemma}
\begin{proof}
	Call $M$ the number of layers of the multicover. By multiplying the left and right sides of \ref{first multicover} by $\charfun{V}$ and integrating over $X$ we get
	\begin{equation*}
	\sum_{i∈I} μ(U_i ∩ V) ⩽ M\, μ(V).
	\end{equation*}
	If the inequality \eqref{basic lemma} did not hold for any $i∈I$, we would then have
	\begin{equation*}
	M \, μ(V) ⩾ \sum_{i∈I} μ(U_i ∩ V) > \frac{1}{1-δ}\, μ(V) \sum_{i∈I}μ(U_i) \stackrel{\mathrm{\ref{second multicover}}}{⩾} M\, μ(V),
	\end{equation*}
	which is a contradiction.
\end{proof}

\begin{lemma}\label{first multicover lemma}
	Let $§{U_n}_{n∈\nn}$ be a multicover with error $δ$. For any integer $k >0$, there is a subset $J⊆\nn$ such that
	\begin{enumerate}[label=(\alph*)]
		\item \label{first subcover} $\displaystyle μ§( \bigcup_{j∈J} U_j)[Big] ⩾ (1-δ)§(1-\frac 1 k)$,
		\item \label{second subcover}  $\displaystyle \sum_{j∈J} μ(U_j) ⩽ k$.
	\end{enumerate}
\end{lemma}
\begin{proof}
	We define a sequence of indices $j_0, j_1,…$ recursively, starting with $j_0 = 0$. Suppose we have defined the first $n$ indices. If $μ§( \bigcup_{\ell <n} U_{j_\ell}) ⩾ (1-δ)§(1- 1/k)$, we stop, otherwise we apply the previous lemma with $V = \bigcup_{\ell <n} U_{j_\ell}$ and find an index $j$ such that
	\begin{equation}\label{proof multicover}
	μ§(U_j ∩ V) ⩽ \frac{1}{1-δ}\, μ(U_j)\, μ(V) < §(1-\frac{1}{k})[Big]\, μ(U_j).
	\end{equation}
	We can now define $j_n$ to be the smallest such $j$. 
	Notice that $j_n$ cannot be equal to any of the previous $j_\ell$'s, because in that case we would have 
	\begin{equation*}
		μ(U_{j_n}) = μ§(U_{j_n} ∩ \bigcup_{\ell < n} U_{j_\ell} )[Big] < §(1-\frac{1}{k})[Big]\, μ(U_{j_n}),
	\end{equation*}
	which is impossible. Also, $j_n$ cannot be smaller than $j_{n-1}$, as this would imply
	\begin{equation*}
		μ§(U_{j_n} ∩ \bigcup_{\ell < n-1} U_{j_\ell} )[Big] 
		⩽ μ§(U_{j_n} ∩ \bigcup_{\ell < n} U_{j_\ell} )[Big]
		< §(1-\frac{1}{k})[Big]\, μ(U_{j_n}),
	\end{equation*}
	which contradicts the minimality of $j_{n-1}$. This means that the sequence $(j_n)_n$ is strictly increasing.

	Denote by $J$ the set of the indices obtained with this construction. 
	If $J$ is finite, then \ref{first subcover} is trivially satisfied. Suppose now that $J$ is infinite and that \ref{first subcover} does not hold. We may apply the previous lemma again to find a $j$ such that \eqref{proof multicover} holds with $V = \bigcup_{i∈J} U_i$.	As $J$ is infinite, there must be an $n$ such that $j < j_{n}$. If we consider just the first $n$ sets $U_{j_\ell}$, we get
	\begin{equation*}
	μ§(U_j ∩ \bigcup_{\ell < n} U_{j_\ell} )[Big] ⩽ μ§(U_j ∩ V)  < §(1-\frac{1}{k})[Big]\, μ(U_j),
	\end{equation*}
	but, by construction, $j_n$ was the smallest index $j$ satisfying this inequality, a contradiction which shows that \ref{first subcover} is satisfied.
	
	To prove \ref{second subcover}, define for each $j∈J$, the set $V_j = U_j \setdiff \bigcup_{i ∈ J,i<j} U_i$, so that the $V_j$'s are pairwise disjoint. Their measure is 
	\begin{equation*}
	μ(V_j) = μ(U_j) - μ§(U_j ∩ \bigcup_{i ∈ J,i<j} U_i)[Big] ⩾ \frac 1k \,μ(U_j),
	\end{equation*}
	by construction. Then, using that fact that the $V_j$'s are pairwise disjoint,
	\begin{equation*}
	\sum_{j∈J} μ(U_j) ⩽ \sum_{j∈J} k\, μ(V_j) = k\, μ§(\bigcup_{j∈J} V_j)[Big] ⩽ k. \qedhere
	\end{equation*}
\end{proof}

The following corollary shows that by relaxing the second condition in the lemma above, we may demand that the set of indices $J$ be finite. 

\begin{corol}\label{finite subcover}
	Let $§{U_n}_{n∈\nn}$ be a multicover with error $δ$. For any integer $k >0$, there is a finite subfamily $§{V_n}_{n⩽N}$ of $§{U_n}_{n∈\nn}$ such that
	\begin{enumerate}[label=(\alph*)]
		\item \label{first finite subcover}$\displaystyle μ§( \bigcup_{n⩽N} V_n)[Big] ⩾ (1-δ)§(1-\frac 1 k)$,
		\item \label{second finite subcover}$\displaystyle \sum_{n⩽N} μ(V_n) ⩽ 2k$.
	\end{enumerate}
\end{corol}
\begin{proof}
	Apply the previous lemma to $§{U_n}_{n∈\nn}$ and $2k$ to find a family $§{V_n}_n ⊆ §{U_n}_{n}$ that satisfies \ref{first multicover lemma}.\ref{first subcover} and \ref{first multicover lemma}.\ref{second subcover} (with $k$ replaced by $2k$). 
	Using \ref{first multicover lemma}.\ref{first subcover} and the fact that the sequence $μ§(\bigcup_{i⩽n} V_i)$ is increasing in $n$, we may find an $N$ such that $μ(\bigcup_{n⩽N} V_n) ⩾ (1-δ)§(1- 1/ k)$ and \ref{first finite subcover} is proved. Point \ref{second finite subcover}  follows trivially from \ref{first multicover lemma}.\ref{second subcover}.
\end{proof}

\subsection{Non-singular free actions of amenable groups}

We will now consider a group $G$ acting measurably on a measure space $(X,£F,μ)$. 
A measurable action $G \actson (X,£F,μ)$ is said to be \emph{free} in the measure theoretic sense (or pointwise free) if the set of $x ∈ X$ for which there is some $g≠1_G$ with $gx = x$, i.e., the set where classical freeness fails, has measure zero. 
We will use a slightly different definition that gets rid of points.
\begin{defin}
	A measurable action $G \actson (X,£F,μ)$ is said to be \emph{setwise free} if for every $g∈G\setdiff §{1_G}$ and every $A∈£F$ with positive measure, there is some positive-measure set $B⊆A$ such that $μ(gB \symdiff B) ≠ 0$. 
\end{defin}
Notice that this is the same as asking that for every $g≠1_G$ and every positive element $a$ of the measure algebra associated to $(X,£F,μ)$ (\Cref{def measure algebra}), there is $0<b⩽a$ such that $gb ≠ b$.

When the group $G$ is countable, if the action is setwise free, then it is also pointwise free, since every subset of $§{x∈X: gx=x}$ is fixed by $g$.
The converse need not hold in general, but it is true if the space $(X,£F)$ is \emph{countably separated}, meaning that there is a countable set $£C⊆£F$ such that for any two distinct points in $X$ there is a set $C∈£C$ containing one but not the other. As an example, the unit interval with its Borel $σ$-algebra is countably separated. The interested reader may refer to \cite[343E--343M]{fremlin3i} for further details and more examples.

\begin{fact}
	Let $X$ be a countably separated measure space. If $G$ acts measurably and pointwise freely, then this action is setwise free.
\end{fact}
\begin{proof}
	Let $(A_n)_{n<ω}$ be a sequence witnessing countable separation of $X$ and let $g$ be an element of $G\setdiff§{1_G}$. Then $g^{-1}x ≠ x$ if and only if there is $n<ω$ such that $x∈ A_n \symdiff gA_n$, and thus 
	\begin{equation*}
		§{x∈ X : g^{-1}x ≠ x}* = \bigcup_{n<ω} A_n \symdiff gA_n.
	\end{equation*} 
	As $G \actson X$ is pointwise free, the set on the left has full measure, so there must be an $n$ such that $μ(A_n \symdiff gA_n) > 0$.
\end{proof}

Recall that a \emph{non-singular transformation} $σ$ of $X$ is a bi-measurable bijection of $X$ such that $σ_*μ = μ$, i.e., $μ(A) = 0$ if and only if $μ(σ^{-1}A) =0$ for all $A∈£F$. The non-singular transformations on $X$ form a group that is denoted by $\aut^*(X,£F,μ)$.
In the remainder of this section, we assume that the action $G \actson (X,£F,μ)$ is by non-singular transformations.

Notice that when the action $G \actson X$ is non-singular, we may replace the condition $μ(gB \symdiff B) ≠ 0$ with $gB ∩ B = ∅$ in the definition of setwise freeness. In fact, if $μ(gB \symdiff B) ≠ 0$, then $μ(B\setdiff gB) > 0$ or $μ(gB\setdiff B) > 0$. In the first case we may choose $B' = B \setdiff gB$, and thus $μ(B') > 0$ and $B'∩gB'=∅$. The second case, on the other hand, is equivalent to $μ(B\setdiff g^{-1}B) > 0$ by non singularity, so we may choose $B' = B\setdiff g^{-1}B$ and again $B'$ has positive measure and is disjoint from $gB'$.

\begin{defin}
	For any $H⊆G$ and $A∈£F$, we say that $A$ is an \emph{$H$-set} if the images $§{hA: h∈H}$ of $A$ under $H$ are pairwise disjoint. 
\end{defin}

\begin{lemma}
	If $G$ acts non-singularly and setwise freely on $X$, then for any finite $H⊆G$ and any $A∈£F$ of positive measure, there is an $H$-set $B⊆A$ of positive measure.
\end{lemma}
\begin{proof}By induction on the number of elements of $H$. If $H=∅$, just take $B= A$, otherwise take an $h∈H$, and let $A'⊆A$ be an $(H\setdiff§{h})$-set of positive measure. As noted above, by setwise freeness and non-singularity, there is some $B ⊆ A'$ of positive measure which is a $§{h}$-set, and thus an $H$-set.
\end{proof}

The following corollary generalises \cite[II.\S 2 Lemma 5]{o-w} to non-singular actions and arbitrary probability spaces.

\begin{lemma}[partition in $H$-sets]\label{multicover partition}
	If $G$ acts non-singularly and setwise freely on $X$, then, for any finite $H⊆G$, there exists a partition $(A_n)_{n∈\nn}$ of $X$ into $H$-sets. Moreover, the family $§{HA_n}_n$ is a multicover of $X$ with $\card{H}$ layers and no error.
\end{lemma}
\begin{proof}
	We construct a sequence of $H$-sets $A_{α}$ of positive measure by recursion on countable ordinals $α$. Suppose that we already have $A_{β}$ for all $β$ less than some countable $α$, and let $B_{α} = X \setdiff \bigcup_{β<α} A_{β}$. If $μ(B_{α})=0$, we can stop. If not, by the previous lemma, there is an $H$-set $A_{α}⊆B_{α}$ of positive measure. Notice that this construction must stop at some countable $α$, since $X$ is of finite measure. Up to re-indexing the sequence $(A_{β})_{β<α}$, we have thus a partition $(A_n)_{n∈\nn}$ of $X$ into $H$-sets.
	
	For the second part of the statement, it suffices to notice that
	\begin{equation*}
		\sum_{n∈\nn} \charfun{HA_n} 
		= \sum_{n∈\nn} \sum_{h∈H} \charfun{hA_n} 
		= \sum_{h∈H} \charfun{h\bigcup_n\!A_n} 
		%= \sum_{h∈H} \charfun{X}
		= \card{H}
	\end{equation*} 
	and the same with $\charfun{(\,⋅\,)}$ replaced by $μ(\,⋅\,)$.
\end{proof}

Suppose now that $G$ is a countable group.

\begin{defin}
	If  $H$ and $K$ are finite subsets of $G$, we call the \emph{$K$-boundary} of $H$ the set $\partial_{K} H$ of those $g∈G$ such that $Kg$ contains points both in $H$ and in $G\setdiff H$. We say that the set $\interior_K H = H \setdiff \partial_K H$ is the \emph{$K$-interior} of $H$. 
\end{defin}

\begin{defin}
	A finite subset $H⊆G$ whose $K$-boundary has fewer than $ε\, \card{H}$ points is said to be \emph{$(K,ε)$-invariant}.
	We say that $G$ is \emph{amenable} if for any $ε>0$ and each finite $K⊆G$, there is a finite $(K,ε)$-invariant $H⊆G$.
\end{defin}

\begin{lemma}\label{stay whithin boundaries}
	Let $H$ and $K$ be finite subsets of $G$. If $K$ is symmetric and contains the identity, then
	\begin{equation*}
	KH = H ∪ \boundary_K H.
	\end{equation*}
\end{lemma}
\begin{proof}
	We have $H⊆ KH$ because $1_G∈ K$, and $\boundary_K H ⊆ KH$ by symmetry of $K$, so $H ∪ \boundary_K H ⊆ KH$.
	For the other inclusion, suppose $kh∉ H$ for some $k∈K$ and $h∈H$. By symmetry, $Kkh \ni k^{-1} k h = h ∈ H$, so $Kkh$ has points in $H$. As $1_G ∈ K$, if $Kkh⊆ H$, we would have $kh = 1_G kh ∈ Kkh⊆ H$, contradicting the assumptions, so $Kkh$ must have points in $G\setdiff H$ too, thus implying that $kh ∈ \boundary_K H$.
\end{proof}

We say that a finite subset $T⊆G$ is a \emph{tile} of $G$ if there exists a set $C⊆G$ that makes the family $§{Tc: c∈C}$ a partition of $G$.

\begin{theo}\label{mono rokhlin}
	Let $G$ be a countable amenable group, and let $0<ε<1$ and $0<δ<ε^3/250$. Suppose that $G$ acts non-singularly and setwise freely on $X$, and that $K⊆G$ is finite, symmetric, and contains the identity.
	If $G$ admits a $(K,δ)$-invariant tile $T$, then there is a $T$-set $W⊆X$ such that $μ(TW) > 1- ε$ and $μ§((\boundary_KT)W) < ε$.
\end{theo}
\begin{proof}
	As $\sqrt[3]{2δ} < ε/5$, we can find an integer $k$ satisfying 
	\begin{equation*}
		\frac{4}{ε} <k < \frac{1}{\sqrt[3]{2δ}}.
	\end{equation*}
	We then set $S= TT^{-1}$. As $G$ is amenable, there exists a finite $H⊆G$ which is $§(S, \frac{1}{2 k^{3}} )$-invariant. 
	By \Cref{multicover partition}, we may find a partition $§{A_n}_n$ of $X$ into $KTT^{-1}H$-sets, 
	and thus $§{HA_n}_{n∈\nn}$ is a multicover with $\card{H}$ layers and no error.  We shall call the sets $HA_n$ the \emph{towers} based on the $A_n$'s, and refer to the set $(\boundary_{S} H)A_n$ as the \emph{boundary} of $HA_{n}$. When we want to speak of sets of the form $TW$ or  $(\boundary_K T)W$, with $W$ a $T$-set, we will use the more explicit terms \emph{$T$-tower} and \emph{$K$-boundary}.

	As $T$ is a tile of $G$, there exists a set $C_0$ such that $§(Tc: c∈C)$ is a partition of $G$. As we are only interested in those $c∈C_0$ that translate $T$ inside of $H$, we will consider the subset $C = §{c∈C_0: Tc⊆ H}$.	Notice that $TC$ covers all of the $S$-interior of $H$. Indeed, if some $h∈H$ is not covered, then there are $t∈T$ and $c∈C_0\setdiff C$ such that $tc = h$. As $Tc\nsubseteq H$, there is a $t'∈T$ such that $t'c ∉ H$, but $c= t^{-1}h$, so $t't^{-1}h \notin H$, witnessing that $h ∈ \boundary_{S}H$.
	
	First we show that the sum of the measures of the boundaries is small compared to $\card H$. As $\boundary_{S} H ⊆ SH$ and the $A_n$'s are $SH$-sets, we have
	\begin{align}\label{small sum of boundary}
	\sum_{n∈\nn} μ§((\boundary_{S} H )A_n) &= \sum_{n∈\nn} \sum_{h∈\boundary_{S} H} μ(h A_n) \notag\\
	&=   \sum_{h∈\boundary_{S} H} \xunderbrace{\sum_{n∈\nn} μ(h A_n)}{= μ(hX) = 1} = \card{\boundary_{S} H} < \frac{1}{2k^{3}} \card{H}
	\end{align}
	Now we want to discard those towers whose boundary has large measure with respect to the total measure of the tower. These are the towers with indices in
	\begin{equation*}
	I= §{ n∈\nn : μ§((\boundary_{S} H) A_n) 
		%todo \stackrel{(\star)}{⩾} 
		⩾
		\frac{1}{k^2}\, μ(HA_n) }[Big].
	\end{equation*}
	The sum of the measures of these unwanted towers is small compared to $\card H$. In fact,
	\begin{equation*}
	\sum _{i∈I} μ(HA_i) 
	%todo \stackrel{(\star)}{⩽} 
	⩽
	k^2 \sum_{i∈I} μ§((\boundary_{S}H) A_i) ⩽ k^2 \sum_{n∈\nn} μ§((\boundary_{S}H) A_n) < %k^{2} ⋅ \frac{1}{k^{3}} \card{H} = 
	\frac{1}{2k} \card{H},
	\end{equation*}
	where the last inequality follows from \eqref{small sum of boundary}. 
	In the same fashion, we also remove the $T$-towers having large $K$-boundary, that is, those with indices in
	\begin{equation*}
	J = §{ n∈ \nn\setdiff I : μ§((\boundary_KT) CA_n) ⩾ \frac{2}{k^2}\, μ(TCA_n)  }.
	\end{equation*}
	The sum of the measures of these towers is
	\begin{align*}
	\sum _{j∈J} μ(HA_j) &⩽ \sum_{j∈J} μ(TCA_j) + \sum_{j∈J} μ§((\boundary_{S}H) A_j) \\
	&⩽\frac{k^2}{2} \sum_{n∈\nn} μ§((\boundary_K T) CA_n) + \sum_{j∈J} \frac{1}{k^2} μ(H A_j)\\
	&< \frac{k^2 \, δ}{2} \card{T} \card{C} + \frac{1}{k^2} \card{H}
	⩽ \frac{1}{4k} \card H + \frac{1}{k^2} \card{H} < \frac{1}{2k} \card{H},
	\end{align*}
	by the choice of $k$.
	This means that if we set $L = \nn \setdiff (I∪J)$, the family $§{HA_\ell}_{\ell∈L}$ is a multicover with $\card{H}$ layers and error $1/k$. We may thus apply \Cref{finite subcover} to find $B_0,…,B_N$, elements of $§{A_n}_n$ satisfying
	\begin{enumerate}[label=(\alph*)]
		\item \label{first bn} $\displaystyle μ§( \bigcup_{n⩽N} HB_n)[Big] ⩾ §(1-\frac 1 k)[Big]^2 > 1-\frac 2 k$,
		\item \label{second bn}$\displaystyle \sum_{n⩽N} μ(HB_n) ⩽ 2k$.
	\end{enumerate}
	In particular, this shows that the towers $HB_n$ cover all of $X$ but a small portion that measures less than $2/k$.

	We now use the same construction as in \cite[II.\S 2 Theorem 5]{o-w} and proceed recursively to define an increasing sequence of $T$-sets $W_0,…,W_N$, which will cover most of the union of the towers based on the $B_n$'s. We start with $W_0 = C B_0$. This is a $T$-set and, by what we said in the previous paragraph, it covers all of $HB_{0}$ except at most its boundary.
	
	Suppose now that we have constructed the $T$-set $W_n$ covering all the interiors of the towers $HB_0,…,HB_n$. To get $W_{n+1}$, we first remove all points $x$ from $W_n$ such that $Tx$ overlaps with $TC B_{n+1}$ and then we add all elements of $CB_{n+1}$. More precisely,
	\begin{equation*}
	W_{n+1} = §{ x∈W_n : Tx ∩ TCB_{n+1} = ∅ } ∪ CB_{n+1}.
	\end{equation*}
	This is again a $T$-set. We now show that in passing from $TW_n$ to $TW_{n+1}$ we have only discarded points in the boundary of $HB_{n+1}$. First notice that we cannot lose points lying in the interior of $HB_{n+1}$. In fact, since $TC$ covers all of the $S$-interior of $H$, we have
	\begin{equation*}
	(\interior_{S}\!H )B_{n+1} ⊆ TCB_{n+1} ⊆ TW_{n+1}.
	\end{equation*}
	The points the we have lost are of the form $tx$ for some $t∈T$, and $x∈ W_{n}$ such that $Tx$ overlaps with $TCB_{n+1} ⊆ HB_{n+1}$. This means that there are $s∈T$, $h∈H$ and $y∈B_{n+1}$ such that $s x = h y$, from which it follows that $tx =  ts^{-1} sx = ts^{-1} hy ∈ S H B_{n+1}$, which is precisely $(H ∪ \boundary_{S}H) B_{n+1}$ by \Cref{stay whithin boundaries}, but we have already shown that $tx$ cannot be in the interior, so it must be in the boundary of $HB_{n+1}$.
	
	Finally, we have constructed a $T$-set $W = W_N$ such that $TW$ covers at least the union of the interiors of the towers $HB_n$. 	
	What is left is the portion of $X$ not covered by these towers, which accounts for a total measure of less than $2/k$, as we have seen above, and the union of their boundaries, whose measure is bounded by
	\begin{equation*}
	\sum_{n⩽ N} μ§((\boundary_{S} H) B_n) ⩽ \frac 1{k^{2}} \sum_{n⩽ N} μ(H B_n) ⩽ \frac 2 k,
	\end{equation*}
	where the second inequality follows from \ref{second bn} and the first from the fact that we discarded those towers $HA_n$ that had large boundary. 
	We have thus $μ(TW) > 1-4/k> 1-ε$. Similarly, we check that the $K$-boundary of $TW$ is small:
	\begin{equation*}
	μ§((\boundary_KT)W) ⩽ \sum_{n⩽N} μ§((\boundary_K T) CB_n) ⩽ \frac{2}{k^{2} } \sum_{n⩽ N} μ(H B_n) ⩽ \frac{4}{k} < ε,
	\end{equation*}
	and this concludes the proof.
\end{proof}

This theorem is useful to us only when then group $G$ has a $(K,δ)$-invariant tile for any arbitrarily small $δ>0$ and each finite $K⊆G$, i.e., if $G$ is \emph{monotileable}. This is the case for instance when $G$ is some power of $\zz$. However, we do not know if every amenable group is monotileable. A recent result by Downarowicz, Huczek and Zhang shows that if we allow multiple tiles at once, then we can choose them arbitrarily invariant. This will enable us to extend the previous result to arbitrary amenable groups in a simple way.

\begin{defin}
	A finite family $§{T_1,…,T_n}$ of finite subsets of $G$ is called a \emph{tiling family} for $G$ if there exist subsets $C_1,…,C_n ⊆ G$ such that $§(T_i c : c ∈ C_i, i⩽ n)$ is a partition of $G$. 
	The sets $T_i$ are called the \emph{shapes} of the tiling and the $C_i$'s are the \emph{centre sets}.  
\end{defin}

\begin{theo}[{\cite[Theorem~4.3]{exact-tiling}}]\label{exact tiling}
	Let $G$ be a countable amenable group. For any $δ>0$ and any finite $K⊆G$, there exists a tiling family for $G$, whose shapes are all $(K,δ)$-invariant.
\end{theo}

We can now adapt the proof of \Cref{mono rokhlin} to the case of a tiling family $£T$ with shapes $T_1,…,T_\ell$ and centre sets $\smash{C_0^{(1)}},…,\smash{C_0^{(\ell)}}$. We start by choosing an integer $k$ in the same way, then we set $S = \bigcup_{i⩽\ell} T_i^{\vphantom{-1}}T_i^{-1}$, and find $H⊆G$ which is $§(S, \frac{1}{2 k^{3}} )$-invariant as before. The sets $A_n$ are now $KT_1^{\vphantom{-1}}T_1^{-1} \cdots T_\ell^{\vphantom{-1}} T_\ell^{-1}$-sets, and we define $C_i = §{c∈ \smash{C_0^{(i)}} : T_i c ⊆ H}$ for each $i⩽\ell$, so the union $\bigcup_i T_iC_i$ still covers the $S$-interior of $H$.

When removing the towers with large boundaries, we define the index set $I$ the same way, whereas $J$ will now be the union of the sets
\begin{equation*}
J_i = §{ n∈ \nn\setdiff I : μ§((\boundary_KT_i) C_iA_n) ⩾ \frac{2}{k^2}\, μ(T_iC_iA_n)  }[Big].
\end{equation*}
In the final construction, for each $1⩽i⩽\ell$, we define $\smash{W_0^{(i)}} = C_iB_0$ and
\begin{equation*}
W_{n+1}^{(i)} = §{ x∈W_n^{(i)} : T_ix ∩  \bigcup_{j⩽\ell}T_jC_jB_{n+1} = ∅ }[Big] ∪ C_iB_{n+1},
\end{equation*}
so that $T_i W_{n+1}^{(i)}$ is disjoint from $T_j W_{n+1}^{(j)}$ for different $i$ and $j$. We then conclude in the same way, \textit{mutatis mutandis}, and thus obtain the following result.

\begin{theo}[decomposition of non-singular free actions]\label{amenable rokhlin lemma}
	Let $G$ be a countable amenable group, and let $0<ε<1$ and $0<δ<ε^3/250$. Suppose that $G$ acts non-singularly and setwise freely on $X$, and that $K⊆G$ is finite, symmetric, and contains the identity.
	If $G$ admits a tiling family $£T$ whose shapes are $(K,δ)$-invariant, then for each $T∈£T$, there exists a $T$-set $W_T⊆X$, and
	\begin{enumerate}
		\item\label{am rokh 1} if $T_1,T_2 ∈ £T$ are different, then $T_1W_{T_1}$ and $T_2W_{T_2}$ are disjoint,
		\item\label{am rokh 2} $μ§(\bigcup_{T∈£T} T W_T) > 1-ε$,
		\item $\sum_{T∈£T} μ§((\boundary_K T) W_T) < ε$.
	\end{enumerate}
\end{theo}

Notice that the invariance of the shapes is only used to ensure that the $K$-boundaries are small. You can see this by taking $K = §{1_G}$ and any $δ$: in this case invariance of the shapes and smallness of the boundaries are trivially satisfied, but we still have the $T$-sets covering almost all of $X$.

\begin{obs}\label{stay within image boundaries}
	Suppose $g∈ \boundary_K T \setdiff T$, for some $T∈£T$. Then $gW_T$ is included in the complement of $\bigcup_{T∈£T} (\interior_K T) W_T$. To see this, take $S∈ £T$, possibly equal to $T$, and $s∈ \interior_K S$, and suppose $μ(gW_T ∩ s W_S) >0$. As $g∈ \boundary_K T$, there is some $k∈K$ such that $kg ∈ T$, but then $μ(kgW_T ∩ ks W_S) >0$ and $ks ∈ S$. This is only possible if $T=S$ and $g = s$, which is not the case. 
\end{obs}

We now show that a non-singular action satisfying the decomposition property of \Cref{amenable rokhlin lemma} must be pointwise free. In the case of a countably separated space, this means that the decomposition property is equivalent to the action being (pointwise and setwise) free.

\begin{propos}
	Suppose that $G$ acts non-singularly on $X$ and that points 1 to 3 above are satisfied for all $K$, $ε$, $δ$, and $£T$ as in the previous theorem. Then the action $G \actson X$ is pointwise free.
\end{propos}
\begin{proof}
	Suppose the action is not pointwise free. Then, as $G$ is countable, there is some $g≠1_G$ such that $Y = §{x∈X: gx = x}$ has positive measure. Choose $K=§{g,g^{-1},1_G}$ and  $ε = μ(Y)/3$. Then apply \Cref{exact tiling} to find a tiling set $§{T_1,…,T_\ell}$ of $G$ whose shapes are all $(K,δ)$-invariant.
	
	By hypothesis, for each $i⩽\ell$, there is a $T_i$-set $W_i$ satisfying points 1 to 3 of the previous theorem. In particular, the union of the $K$-interiors of the $T_i$-towers $T_iW_i$ covers all of $X$ except at most a portion of measure $2ε < μ(Y)$. This means that there is a shape $T_i$, an element $t∈\interior_K T_i$, and a point $x∈ W_i$ such that $tx ∈ Y$. Now, $gt$ must belong to $T_i$, otherwise $t$ would be in the $K$-boundary of $T_i$. It follows that $gtx ≠ tx$, since $W_i$ is a $T_i$-set, but $g(tx) = tx$ by definition of $Y$, a contradiction.
\end{proof}

\section{Representation of group actions on $L_p$ lattices}\label{section : representation}

In this section we show that group actions on $L_p$ lattices can be represented as non-singular actions on measure spaces. This is based mainly on results in \cite{fremlin3i}.

We split the representation into two steps by introducing an intermediate kind of structures, namely measure algebras. 
Recall that a \emph{Boolean algebra} $(€A,+,⋅)$ is a ring with multiplicative identity $1_{€A}$, satisfying $a^2 = a$ for all $a∈€A$. Since we often think of Boolean algebras as arising from algebras of sets, we will generally use the alternative notations for the ring operations $a \symdiff b = a+b$ and $a ∩ b = a⋅b$. Following the same principle, we shall write $a ∪ b$ to mean $a+b+ab$, and $a \setdiff b$ for $a + ab$.

A Boolean algebra has a natural order defined by $a ⊆ b$ if $a ∩ b = a$, which makes it a Boolean lattice. 
Some of the results in this section depend on how much this order is complete. A Boolean algebra $€A$ is said to be \emph{Dedekind complete} if every non-empty subset of $€A$ has a least upper  bound (or a greatest lower bound). Sequential Dedekind completeness is defined the same way, except that we only need the countable subsets to have a least upper bound.

\begin{defin}\label{def measure algebra}
	A \emph{measure algebra} $(€A,μ)$ is the aggregate of a sequentially Dedekind complete Boolean algebra $€A$ and a function $μ\colon €A \to \ccint{0,\infty}$ such that
	\begin{enumerate}
		\item for all $a∈€A$, $μ(a) = 0$ if and only if $a = 0$,
		\item for all countable disjoint $B⊆€A$, $\sup_{b∈B} μ(b) = μ(\sup B)$.
	\end{enumerate}
\end{defin}

Given a measure space $(X,£F,μ)$, we can associate a measure algebra to it in a natural way. Consider the ideal $£N⊆£F$ of the negligible (measurable) sets of $(X,£F,μ)$, let $€A$ be the Boolean algebra quotient $£F/£N$, and define $\bar{μ}\colon €A \to \ccint{0,\infty}$ by $\bar{μ}§([A]_{£N}) = μ(A)$. Then, $(€A,\bar{μ})$ is a measure algebra \cite[321H]{fremlin3i}. We will refer to it as the \emph{measure algebra of $(X,£F,μ)$}.

A non-singular transformation $φ$ of $(X,£F,μ)$ induces a Boolean isomorphism $φ^{\natural}$ of $€A$ by 
\begin{equation}\label{induced on algebra}
	φ^{\natural}[A]_{£N} = [φ^{-1}A]_{£N}
\end{equation}
for every $A ∈ £F$. Moreover, the map sending $φ$ to $φ^{\natural}$ is a group homomorphism from $\aut^*(X,£F,μ)$ to $\aut(€A)$, the group of Boolean automorphisms of $€A$. Notice that the measure of the measure algebra plays no role here, since non-singularity simply translates to preservation of $0_{€A}$.

In general, if no additional property is required of the measure space, there may be an automorphism of $€A$ that is not induced by a non-singular transformation of the underlying measure space. As an example, suppose that $X$ is the disjoint union of two non-negligible sets $C$ and $U$, one countable and the other uncountable, and suppose also that $£F= §{∅,C,U,X}$. Then, the measure algebra of $(X,£F,μ)$ is $(£F,μ)$ and the map $σ$ swapping $C$ and $U$ and fixing $∅$ and $X$ is an automorphism of $£F$. However, there cannot be a bijection $φ\colon X \to X$ satisfying $φ^{-1}[U] = C$.

In order to complete the correspondence, we may replace the initial measure space with a “rectified” version of it, with the same measure algebra $€A$ and with the group of non-singular transformations isomorphic to $\aut{€A}$. This can be achieved by exploiting the Stone representation of Boolean algebras.

To this end, consider the Stone space $Z$ associated to $€A$, together with the \tsigma-algebra $Σ$ given by subsets of $Z$ of the form $C \symdiff M$, where $C$ is clopen and $M$ is meagre. Then $€A$ is isomorphic to the quotient of $Σ$ by the ideal $£M$ of the meagre subsets of $Z$. Moreover, if $μ$ is a measure on $€A$, we can define a measure on $(Z,Σ)$ by setting for all $a ∈ €A$, $ν(\widehat{a} \symdiff M) = μ(a)$, where $\widehat{a}$ is the clopen set associated to $a$. The interested reader may refer to \cite[314M, 321J]{fremlin3i} for more details.
In particular, the meagre sets are precisely the negligible sets, so $(€A,μ)$ is isomorphic to the measure algebra of $(Z,Σ,ν)$. We shall denote this measure space by $\stone(€A,μ)$.
The following result then completes the correspondence.

\begin{fact}[{\cite[312P]{fremlin3i}}]
	Let $€A$ be a Boolean algebra, with Stone space $Z$. There is a one-to-one correspondence between Boolean automorphisms $σ$ of $€A$ and homeomorphisms $φ$ of $Z$, given by the formula
	\begin{equation*}
		σ(a) = b \iff φ^{-1}[\widehat{a}] = \widehat{b},
	\end{equation*}
	for all $a,b∈ €A$ and corresponding clopen sets $\widehat{a}$ and $\widehat{b}$.
\end{fact}

This means that for each automorphism $σ$ of $€A$ there is exactly one homeomorphism $φ_{σ}$ of $Z$ satisfying $σ(a) = [φ_{σ}^{-1} \widehat{a}]_{£M}$ for all $a ∈ €A$, or equivalently $σ = (φ_{σ})^{\natural}$, using the notation defined before. Because $φ_{σ}$ is a homeomorphism, images and preimages of meagre sets are meagre, so $φ_{σ}$ is a non-singular transformation of $(Z,Σ,ν)$.
This shows that when the measure space is the Stone space of the algebra, the map $φ \mapsto φ^{\natural}$ is invertible, and so it is an isomorphism. We have thus proved the following result.

\begin{lemma}\label{aut isom algebra stone}
	Let $(€A,μ)$ be a measure algebra. Then 
	\begin{equation*}
		\aut^*(\stone(€A,μ)) \isom \aut(€A)
	\end{equation*}
	via the map $φ \mapsto φ^{\natural}$ defined in \eqref{induced on algebra}.
\end{lemma}

We shall now turn to the representation of Banach-lattice automorphisms as Boolean-algebra automorphisms. Recall that a \emph{Riesz space}, or vector lattice, is a vector space $V$ together with a lattice order $⩽$ that is compatible with the vector structure, i.e., such that for any $x,y,z∈V$ with $x⩽ y$, and any positive $α∈\rr$, we have $x+z⩽y+z$ and $αx ⩽ αy$. The space of $£F$-measurable functions from $X$ to $\rr$ with the usual order is a Riesz space. We will denote it by $L^0(X,£F)$.

Fix $1⩽p<\infty$. Consider now a measure algebra $(€A,μ)$, with Stone representation $(Z,Σ,\tilde{μ})$, and define
\begin{align*}
	L^0(€A) &= L^0(Z,Σ),\\
	L^p(€A,μ) &= L^p(Z,Σ,\tilde{μ})
\end{align*}
The interested reader may find a direct construction of the spaces $L^0(€A)$ and $L^p(€A,μ)$ that does not make any reference to measure spaces in \cite{fremlin3i}, Sections 364 and 366.

The following fact shows that a concrete $L_p$ space only depends on the measure algebra of the underlying measure space.

\begin{fact}[{\cite[366B]{fremlin3i}}] \label{measure vs algebra}
	Let $(X,£F,μ)$ be a measure space, and $(€A,\bar{μ})$ its measure algebra. Then $L^0(X,£F)$ is isomorphic as a Riesz space to $L^0(€A)$, and this isomorphism $Φ$ restricts to a Banach-lattice isomorphism between $L^p(X,£F,μ)$ and $L^p(€A,\bar{μ})$. Moreover, if for any $A∈£F$, we call $A^\bullet$ the class of $A$ modulo negligible sets, then $Φ$ satisfies \begin{math}
		Φ(\charfun A) = \charfun{\widehat{A^\bullet}}.
	\end{math}
\end{fact}

This result could be strengthened: the Banach lattices above being isomorphic is actually equivalent to the measure spaces having measure algebras with isomorphic Dedekind completions.

We will now show a correspondence between Boolean automorphisms of $€A$ and Riesz automorphisms of $L^0(€A)$.
For any $a∈ €A$, we can write $χ\lowersub{a}$ to mean the class of the characteristic function $χ\lowersub{\widehat{a}}$ of the clopen $\widehat{a}⊆Z$ corresponding to $a$.

\begin{fact}[{\cite[364R]{fremlin3i}}]\label{aut of L0}
	Let $σ$ be an automorphism of $€A$. Then there exists a unique multiplicative Riesz automorphism $T_{σ}$ of $L^0(€A)$ satisfying $T_{σ} (χ\lowersub{a}) = χ\lowersub{σ(a)}$. In addition,
	\begin{enumerate}
		\item\label{ts point 1} if $h ∈ L^0(\rr)$, then $T_{σ}(h \comp u) = h \comp T_{σ}(u)$,
		\item the map $σ \mapsto T_{σ}$ defines an isomorphism between $\aut(€A)$ and $\aut(L^0(€A))$.
	\end{enumerate}
\end{fact}

In order to have a correspondence between Boolean automorphisms of $€A$ and  Banach-lattice automorphisms of $L^p(€A,μ)$, we need to include a factor that changes the density of the measure. In the finite measure case, this would be the Radon-Nikodym derivative of the pushforward measure. The following fact guarantees the existence of a generalised version of such derivative to semi-finite measures, provided that the Boolean algebra is Dedekind complete. As with measure spaces, a measure $μ$ for a Boolean algebra $€A$ is said to be  \emph{semi-finite} if for each $a∈€A$ with $μ(a) = \infty$, there is some $b⊆a$ such that $0<μ(b) < \infty$. Here we use the notation $\int_a u \dd μ$ to mean $\int_{\widehat{a}} u \dd \tilde{μ}$, where $\tilde{μ}$ the measure of the Stone representation of $(€A,μ)$.

\begin{fact}[{\cite[365T]{fremlin3i}}]\label{general Radon-Nikodym}
	If $€A$ is a Dedekind complete Boolean algebra and $μ,ν \colon €A \to \ccint{0,\infty}$ are semi-finite measures on $€A$, then there is a unique 
	$u ∈ L^0(€A)$ such that $ν(a) = \int_a u \dd μ$ for every $a∈€A$. Moreover, this element $u$ satisfies the following properties.	 
	\begin{enumerate}
		\item $u$ is strictly positive,
		\item for every $a∈€A$, $μ(a) = \int_a \frac{1}{u} \dd ν$,
		\item for every $v∈ L^1(€A,ν)$, $\int v \dd ν = \int vu \dd μ$.
	\end{enumerate}
\end{fact}
We will denote the element $u$ defined above by $\frac{\dd ν}{\dd μ}$, as usual. Notice also that, if $η$ is another semi-finite measure on $€A$, then $\frac{\dd μ}{\dd ν} ⋅ \frac{\dd ν}{\dd η} = \frac{\dd μ}{\dd η}$, just as in the \tsigma-finite case. 
Now, if $μ$ is a semi-finite measure on a complete Boolean algebra $€A$ and $σ$ is an automorphism of $€A$, then $μσ = μ \comp σ$ is also semi-finite, so that $ \frac{\dd μσ}{\dd μ}$ exists in $L^0(€A)$. 
Additionally, if $w ∈ L^1(€A, μσ)$, then
\begin{equation}\label{law of composition}
	\int w \dd μσ = \int T_{σ}(w) \dd μ,
\end{equation}
which can be shown as usual for simple functions and then extended by taking limits.
If $τ$ is another automorphism of $€A$, then, using the previous identity,
\begin{align*}
	 \int_a \frac{\dd μ στ}{\dd μτ} \dd μτ 
	 = μστ(a) 
	 &= \int_{τa} \frac{\dd μσ}{\dd μ} \dd μ \\
	 &= \int T_{τ}(\charfun{a})⋅ \frac{\dd μσ}{\dd μ} \dd μ %\\
	 %&= \int T_{τ}§(\charfun{a}⋅ T_{τ^{-1}}\frac{\dd μσ}{\dd μ}) \dd μ \\
	 %&= \int \charfun{a}⋅ T_{τ^{-1}}\frac{\dd μσ}{\dd μ} \dd μτ
	 = \int_a T_{τ^{-1}}\frac{\dd μσ}{\dd μ} \dd μτ,
\end{align*}
for all $a∈€A$, whence
\begin{equation}\label{rn derivative composition}
	T_{τ}§( \frac{\dd μ στ}{\dd μτ} ) = \frac{\dd μσ}{\dd μ},
\end{equation}
by \Cref{general Radon-Nikodym}.

We shall now recall some definitions and basic results concerning bands in Banach lattices; the interested reader may find more details on the subject in \cite[\S 1.2]{meyer-banach}. Let $E= L^p(€A,μ)$ be a Banach lattice, and let $A$ be a subset of $E$. The \emph{disjoint complement} of $A$ is defined by
\begin{equation*}
	A^{\perp} = §{ u ∈ E : \abs{u} \meet \abs{w} =0, \text{ for every } w∈A},
\end{equation*}
and the \emph{band generated} by $A$ is the double complement 
\begin{equation*}
	\band{A} = (A^{\perp})^{\perp}.
\end{equation*}
The sets $A^\perp$ and $\band{A}$ are vector sublattices of $E$, and $E = \band{A} \oplus A^{\perp}$. In particular, there is a linear projection $P_{\band{A}} \colon E \mapsto \band{A}$ that gives the component of an element in $\band{A}$. When $A=§{w}$, we will use the notation $\band{w}$ instead of $\band{§{w}}$, and define 
\begin{equation*}
	u \restr w = P_{\band{w}}(u).
\end{equation*}
Notice that if $T$ is a Riesz space automorphism, then $T(u\restr w) = Tu \restr Tw$.

Another relevant fact is that the collection of all bands of $E$, i.e., sublattices of the form $\band{A}$, together with lattice operations $B \meet C = B ∩ C$, $B \join C = \band{B + C}$ and $\neg B = B^{\perp}$, is a Dedekind-complete Boolean algebra, which is called the \emph{band algebra} of $E$. 
The following result enables us to identify the band algebra of $L^p(€A,μ)$ with the underlying algebra $€A$, under certain conditions.

\begin{fact}[{\cite[366X-b, 364O]{fremlin3i}}]	\label{band algebra isom}
	If $(€A,μ)$ is a Dedekind-complete and semi-finite measure algebra, then $€A$ is isomorphic to the band algebra of $L^p(€A,μ)$ via the map $a \mapsto \band{\charfun{a}}$.
\end{fact}
Notice that even when $\charfun{a}$ does not belong to $L^p(€A,μ)$, we may consider its band $B$ in $L^0(€A)$, and $B∩ L^p(€A,μ)$ is indeed a band of $L^p(€A,μ)$, which justifies the notation $\band{\charfun{a}}$ used in the statement here above. Moreover, under the identification given by this fact, 
\begin{equation*}
	u \restr w = u ⋅\charfun{\band{w}}, 
\end{equation*}
for all $u,w ∈ L^p(€A,μ)$, with the product evaluated in $L^0(€A)$.
Another tool that we will use is the following immediate corollary of \cite[364L and 368C]{fremlin3i}. 

\begin{fact}\label{extension to measurable}
	If $(€A,μ)$ is a semi-finite measure algebra, then every Banach lattice automorphism of $L^p(€A,μ)$ extends uniquely to a Riesz space automorphism of $L^0(€A)$.
\end{fact}

Define now, for each $σ ∈ \aut(€A)$ and each $u$ in $L^p(€A,μ)$,
\begin{equation}\label{induced transformation}
	\widetilde{σ}(u) = §(\frac{\dd μσ^{-1}}{\dd μ})^{\frac{1}{p}} ⋅ T_{σ}(u).
\end{equation}
First notice that $\widetilde{σ}$ is a Riesz automorphism of $L^p(€A,μ)$, because $T_{σ}$ is, and $\dd μσ/\!\dd μ$ is strictly positive. 
Secondly, by \Cref{aut of L0}-1
, we have $T_{σ}(\abs{u})^p = T_{σ}(\abs{u}^p)$, and thus
\begin{equation*}
\norm{\widetilde{σ}(u)}^p 
= \int \frac{\dd μσ^{-1}}{\dd μ}⋅ \abs{T_{σ}(u)}^p \dd μ 
= \int T_{σ}(\abs{u}^p) \dd μσ^{-1} = \int \abs{u}^p \dd μ = \norm{u}^p,
\end{equation*}
where we have used \eqref{law of composition} for the third equality. This shows that $\widetilde{σ}$ is in fact a Banach-lattice automorphism of $L^p(€A, μ)$.
We can now prove the final piece needed for the representation of group actions on a Banach $L_p$ lattice.

\begin{lemma}\label{aut isom algebra Lp}
	Let $(€A,μ)$ be a Dedekind-complete and semi-finite measure algebra. Then 
	\begin{equation*}
		\aut(€A) \isom \aut(L^p(€A,μ)) 
	\end{equation*}
	via the map $σ \mapsto \widetilde{σ}$ defined in \eqref{induced transformation}.
\end{lemma}
\begin{proof}
	We have already proved that $σ \mapsto \widetilde{σ}$ is a well defined map $Φ\colon \aut(€A) \to \aut(L^p(€A,μ))$. Given $σ,τ ∈ \aut(€A)$, we have
	\begin{equation*}
		\frac{\dd μτ^{-1}σ^{-1}}{\dd μ} = \frac{\dd μτ^{-1}σ^{-1}}{\dd μσ^{-1}} ⋅ \frac{\dd μσ^{-1}}{\dd μ} = T_{σ}§( \frac{\dd μτ^{-1}}{\dd μ} )⋅ \frac{\dd μσ^{-1}}{\dd μ}
	\end{equation*}
	by \eqref{rn derivative composition}, and thus
	\begin{align*}
		\widetilde{στ}(u) &= §(\frac{\dd μτ^{-1}σ^{-1}}{\dd μ})^{\frac{1}{p}} ⋅ T_{στ}(u) \\
		&= §(\frac{\dd μσ^{-1}}{\dd μ})^{\frac 1p}  ⋅ T_{σ}§( §(\frac{\dd μτ^{-1}}{\dd μ})^{\frac 1p} ⋅ T_{τ}(u) )  
		= \widetilde{σ}(\widetilde{τ}(u)),
	\end{align*}
	since the map $σ \mapsto T_{σ}$ is a multiplicative isomorphism between $\aut(€A)$ and $\aut(L^0(€A))$, by \Cref{aut of L0}. 
	
	If $σ ∈ \aut(€A)$, apply \Cref{extension to measurable} to extend $\widetilde{σ}$ to $L^0(€A)$. This lets us compute $\widetilde{σ}(\charfun{1_{€A}})$, which is precisely the density factor $(\dd μσ^{-1} / \dd μ)^{1/p}$. As a consequence, if $σ,τ ∈ \aut(€A)$ are such that $\widetilde{σ} = \widetilde{τ}$, then $T_{σ} = T_{τ}$, and thus $σ = τ$. 
	This means that $Φ$ is injective. 
	
	Consider now some $T ∈ \aut(L^p(€A,μ))$, it only remains to prove that there is $σ∈\aut(€A)$ such that $\widetilde{σ} = T$. First use \Cref{extension to measurable} to extend $T$ to $L^0(€A)$, then use \Cref{band algebra isom} to identify $€A$ with the band algebra of $L^p(€A,μ)$. At this point we may define
	\begin{equation*}
		σ\colon a \mapsto \band{T\charfun{a}}
	\end{equation*}
	and show that $σ∈\aut(€A)$. In order to prove that $\widetilde{σ} = T$, we just need to check that for all $a,b ∈ €A$ with $μb < \infty$, we have 
	\begin{equation*}
		\int_a \widetilde{σ}(\charfun{b})^p \dd μ = \int_a T(\charfun{b})^p\dd μ.
	\end{equation*}
	The left-hand side is equal to $\int_a \charfun{σb} \dd μσ^{-1} = μ(b ∩ σ^{-1} a)$, and the right-hand side is
	\begin{equation*}
		\int §(T(\charfun{b}) \restr \charfun{a})^p \dd μ 
		= \int §(T(\charfun{b}\restr T^{-1}\charfun{a}) )^p \dd μ 
		= \int §(T(\charfun{b} ⋅ \charfun{σ^{-1}a}) )^p \dd μ
		%= \int §(\charfun{b} ⋅ \charfun{σ^{-1}a})^p \dd μ 
		= μ(b ∩ σ^{-1}a),
	\end{equation*}
	which completes the proof.
\end{proof}

We will now introduce the notion of $G$-$L_p$ lattice, and put the results of this section together to prove that each such structure can represented by a concrete $L_p$ lattice over a measure space equipped with a non-singular group action. This is a generalisation of Kakutani's Representation Theorem \cite{kakutani-L-representation}, which can be rephrased with the language of measure algebras as follows.

\begin{fact}[Kakutani's Representation Theorem]\label{kakutani}
	Let $1⩽ p < \infty$, $E$ an $L_p$ lattice, and $E_0⊆E$ a maximal disjoint subset of positive elements of $E$. Then, there exists a Dedekind-complete and semi-finite measure algebra $(€A,μ)$ such that $E$ is isomorphic to $L^p(€A,μ)$ as a Banach lattice, and the elements of $E_0$ are mapped to characteristic functions under this isomorphism.
\end{fact}
\begin{proof}
	See the proof of \cite[Theorem~2.7.1]{meyer-banach}, and apply \cite[322N]{fremlin3i} and \Cref{measure vs algebra}.
\end{proof}

\begin{defin}
	Given $1⩽p<\infty$ and a group $G$, we call a \emph{$G$-$L_p$ lattice} the aggregate of an $L_p$ lattice $E$ and a homomorphism $ρ\colon G \to \aut(E)$ defining a group action of $G$ on $E$ by Banach-lattice automorphisms. 
	
	A \emph{homomorphism $Φ\colon(E,ρ) \to (E',ρ')$ of $G$-$L_p$ lattices} is an isometric normed lattice homomorphism $Φ\colon E \to E'$ such that $Φ \comp ρ(g) = ρ'(g) \comp Φ$, for each $g∈G$.
\end{defin}

\begin{defin}
	Suppose $π\colon G \to \aut^*(X,£F,μ)$ is a non-singular action on a measure space $(X,£F,μ)$. We define $L^p(X,£F,μ,π)$ to be the $G$-$L_p$ lattice $(E,ρ)$, where $E = L^p(X,£F,μ)$ and $ρ(g) = \widetilde{(πg)^{\natural}}$, as defined in \eqref{induced transformation} and \eqref{induced on algebra}, for each $g∈G$.
\end{defin}

\begin{defin}
	An element $u$ of an $L_p$ space $E$ is said to be a \emph{weak unit} if $\band{u} = E$, or equivalently, if there is no non-zero $w∈E$ disjoint from it.
\end{defin}

\begin{theo}[representation of $G$-$L_p$ lattices] \label{representation theorem}
	Let $(E,ρ)$ be a $G$-$L_p$ lattice and $E_0 ⊆ E$ a maximal disjoint subset of positive elements of $E$. Then there is a measure space $(Z,Σ,μ)$, and a non-singular action $π$ of $G$ on it, such that $(E,ρ)$ is isomorphic to $L^{p}(Z,Σ,μ,π)$ as a $G$-$L_p$ lattice.
	
	Moreover, this isomorphism can be chosen so that the elements of $E_0$ correspond to characteristic functions on $Z$. In particular, if $E$ has a weak unit $u$, then $Z$ can be chosen of total measure $\norm{u}^p$ with $u$ corresponding to $χ\lowersub{Z}$.
\end{theo}
\begin{proof}
	By \Cref{kakutani} and \Cref{measure vs algebra}, we may assume that $E = L^{p}(Z,Σ,μ)$, where $(Z,Σ,μ)$ is the Stone representation of  some Dedekind-complete and semi-finite measure algebra $(€A,\bar{μ})$, and that the elements of $E_0$ are characteristic functions of clopen subsets of $Z$.
	
	The moreover part of the theorem is then satisfied. In particular, when $E$ has a weak unit $u$, we can choose $E_0= §{u}$, so that $u$ is of the form $\charfun{A}$ for some clopen $A⊆Z$. By maximality of $E_0$, $A$ must be $Z$, and thus $μ(Z) = \norm{\charfun{Z}}^p = \norm{u}^p$.

	We shall now deal with the action of $G$. By \Cref{aut isom algebra stone} and \Cref{aut isom algebra Lp}, the map $Ψ\colon \aut^*(Z,Σ,μ) \to \aut(E)$ sending $φ$ to $\smash{\widetilde{(φ^\natural)}}$ is a group isomorphism, which means that $π = Ψ^{-1} \comp ρ$ is an action of $G$ on $(Z,Σ,μ)$ by non-singular transformations. By definition, the action of $G$ in the concrete $G$-$L_p$ lattice $L^p(Z,Σ,μ,π)$ is precisely $Ψ \comp π = ρ$, which concludes the proof.
\end{proof}

\begin{obs}
	Let $(E,ρ)$ be a $G$-$L_p$ lattice, and suppose that $F⊆E$ is a closed vector sublattice of $E$ that is invariant under $ρ$. Then we can define $ρ_F(g) = ρ(g) \restr F$ for all $g∈G$, and $(F,ρ_F)$ is a sub-$G$-$L_p$-lattice of $(E,ρ)$.
	
	By the previous theorem, we may assume $(E,ρ)= L^p(Z,Σ,μ,π)$. Consider now the family
	\begin{equation*}
		Σ_{F} = §{ \supp (f) \symdiff M : f ∈ F, M \text{ meagre} },
	\end{equation*}
	where $\supp (f)$ is clopen associated to image of $\band{f}$ under the isomorphism of \Cref{band algebra isom}, between the band algebra and the underlying measure algebra of $E$. 
	
	Notice that is a sub-\tsigma-algebra of $Σ$ that is invariant under $π$, but $F$ does not necessarily coincide with a subspace of $E$ of the form $L^p(Z',Σ_F,μ)$ for some $Z' ⊆ Z$. For instance, if $τ \colon L^p\ccint{0,1} \to L^p\ccint{1,2}$ is induced by a translation of the real line, then the set 
	\begin{equation*}
		F = §{a + \frac 12 τ(a) : a ∈ L^p\ccint{0,1} }[Big] ⊆ L^p\ccint{0,2}
	\end{equation*}
	is a sublattice of $L^p\ccint{0,2}$, but not of the form $L^p(Z',Σ_F, λ)$, for any $Z' ⊆ \ccint{0,2}$.
\end{obs}

\section{Free $G$-$L_p$ lattices and their model theory}\label{GLp lattices}

In this section we introduce a notion of freeness for $G$-$L_p$ lattices, which corresponds to setwise freeness under the isomorphism of \Cref{representation theorem}. The decomposition lemma (\Cref{amenable rokhlin lemma}) for probability spaces then yields a local decomposition of a free $G$-$L_p$ lattice in disjoint bands in a natural way. This will allow us to axiomatise the class of free $G$-$L_p$ lattices in the continuous language of Banach lattices, and prove that its theory is model complete, thus generalising a result in \cite{scielzo}.

Let $G$ be a countable group, and $1⩽p<\infty$. In the following, when referring to $G$-$L_p$ lattices, we will simply write $E$ and $gv$ instead of $(E,ρ)$ and $ρ(g)(v)$.
\begin{defin}
	A $G$-$L_p$ lattice $E$ is said to be \emph{functionally free} if for all $g∈G\setdiff §{1_G}$ and all positive $u∈ E$ there is a positive $v ∈ \band{u}$ such that $v$ and $gv$ are disjoint.
\end{defin}

Suppose now that $E$ is a concrete $G$-$L_p$ lattice $L^p(X,£F,μ,π)$ over a semi-finite measure space $(X,£F,μ)$. If $π$ is setwise free, $g∈G\setdiff{1_G}$, and $u$ is a positive element of $E$, then there is a set $A⊆\supp u$ of positive measure which is disjoint from $gA$. By semi-finiteness, we may find $B⊆A$ of positive and finite measure, so $v = \charfun{B}$ witness that $E$ is functionally free. In the same fashion, one can show that the converse holds as well. We have thus the following fact.

\begin{fact}
	Let $(X,£F,μ)$ be a semi-finite measure space. A $G$-$L_p$ lattice $L^p(X,£F,μ,π)$ is functionally free if and only if $π$ is setwise free.
\end{fact}

\subsection{Decomposition of free $G$-$L_p$ lattices}

We will now translate \Cref{amenable rokhlin lemma} into the language of Banach lattices. Given an element $v$ in a $G$-$L_p$ lattice $E$ with a weak unit, we will find an array of disjoint elements $tu_T$, with $T$ varying in a tiling family $£T$ of $G$ and $t∈T$, which cover most of the support of $v$. We may even choose the $u_T$'s to be in a prescribed $G$-invariant sublattice of $E$. Furthermore, if the shapes in $£T$ are sufficiently invariant, we may also choose the $u_T$'s so that their images $t u_T$, with $t$ in the boundary of $T$, only cover a small portion of $v$.

Here we use the notations $\inboundary{K} T = \boundary_K T ∩ T$ for the \emph{inside $K$-boundary} of $T$, and $E_+$ for the set of positive elements of $E$. Recall also that if $v,u∈ E$, we denote by $v \restr  u$ the projection of $v$ onto the band $\band{u}$.

\begin{lemma}[decomposition of $G$-$L_p$ lattices]\label{rokh lemma}
	Let $G$ be a countable amenable group, and let $K$, $ε$, $δ$, and $£T$ be as in \Cref{amenable rokhlin lemma}.
	Suppose that $E$ is a functionally free $G$-$L_p$ lattice and that $F⊆E$ is a $G$-invariant closed sublattice containing a weak unit $w_0$ for $E$.
	For any $v ∈ E_+$, there are elements $u_{\mathrm{err}}∈F_+$ and $u\lowersub{T}∈F_+$ for each $T∈£T$,  such that
	\begin{enumerate}
		\item $§{ u_{\mathrm{err}} } ∪ §{ t u_T : t∈T∈£T }$ is a maximal disjoint subset of $F$,
		\item $\norm{ v \restr  u_{\mathrm{err}} } < ε^{1/p} \norm {w_0 + v}$,
		\item $\displaystyle \sum_{T∈£T}\, \sum_{s∈\inboundary{K} T} \norm{ v\restr s u_T} ⩽  ε^{1/p} \norm {w_0 + v}$. 
	\end{enumerate}
\end{lemma}
\begin{proof}
	First notice that $w = w_0 + v$ is a weak unit of $E$ as well. 
	By the representation \Cref{representation theorem}, we may assume $E$ to be the $G$-$L_p$ lattice over a measure space $(Z,Σ,ν)$ with total measure $ν(Z) = \norm{w}^p$ and $w = \charfun{Z}$.
	
	We now replace $ν$ with the equivalent probability measure $μ= ν/ν(Z)$ and apply \Cref{amenable rokhlin lemma} to $G \actson (Z,Σ_F,μ)$, where $Σ_F$ is the collection of the supports of the elements of $F$, which is a \tsigma-algebra.
	This yields, for each $T∈ £T$, some $T$-set $W_T ∈ Σ_F$ satisfying the following conditions:
	\begin{enumerate}
		\item if $T,S ∈ £T$ are different, then $TW_{T}$ and $SW_{S}$ are disjoint,
		\item $ν§(\bigcup_{T∈£T} T W_T) > ν(Z)(1-ε) ⩾ ν(Z) - ε\norm{w}^p$,
		\item $\sum_{T∈£T} ν§((\boundary_K T) W_T) < εν(Z) ⩽ ε\norm{w}^p$.
	\end{enumerate}
	By definition of $Σ_F$, there is some $u_{\mathrm{err}} ∈ F$ whose support is $W_{\mathrm{err}} = Z \setdiff \bigcup_{T∈£T} T W_T$, and for each $t∈T∈£T$, there is $u_{T,t} ∈ F$ whose support is $tW_T$. The element
	\begin{equation*}
		u_T = \sum_{t∈T} t^{-1} u_{T,t} 
	\end{equation*}
	has support equal to $W_T$, so that $§{ u_{\mathrm{err}} , t u_T : t∈T∈£T }$ is a maximal disjoint subset of $F$. 
	For the second point, notice that $v \restr u_{\mathrm{err}} ⩽ w = χ\lowersub{Z}$, hence 
	\begin{equation*}
		\norm{v \restr u_{\mathrm{err}} }^p ⩽ \norm{χ\lowersub{Z} \restr u_{\mathrm{err}} }^p = ν(W_{\mathrm{err}}) < ε \norm{w}^p.
	\end{equation*}
	Similarly, for the last point, $v \restr s u_T ⩽ χ\lowersub{Z} \restr s u_T =  χ\lowersub{s W_T}$, and because the $su_T$'s with $s∈ \inboundary{K} T$ and $T∈£T$ are pairwise disjoint, we have
	\begin{align*}
		\norm{ \sum_{T∈£T}\, \sum_{s∈\inboundary{K} T}  v\restr s u_T}[Big]^p
		&= \sum_{T∈£T}\, \sum_{s∈\inboundary{K} T} \norm{ v\restr s u_T}^p\\
		&⩽ \sum_{T∈£T}\, \sum_{s∈\inboundary{K} T} ν(sW_T)
		⩽ \sum_{T∈£T} ν§((\boundary_K T) W_T)
		<  ε \norm {w}^p \qedhere
	\end{align*}
\end{proof}

The lemma we have just proved has still some measure theoretical flavour, since the decomposition of the element $v∈E$ is performed at the level of support, as highlighted by the use of restrictions. In order to axiomatise the class of free $G$-$L_p$ lattices, we need a characterisation of functional freeness purely in the language of Banach lattices. 

By sacrificing the conditions on the boundaries, we may consider, for each non-identical $g∈G$, a tile of $G$ made up of powers of $g$. We will then show that by applying the previous lemma to this tile and some $u∈E_+$, we can find an element $v∈E$ disjoint from $gv$ and satisfying $0<v<u$, up to a small error.

\begin{defin}
	Let $g∈G$. If $g$ has finite order we define its \emph{tiling number} $\tilenum(g)$ to be equal to its order. If $g$ has infinite order, we set $\tilenum(g) = 2$.
\end{defin}

\begin{lemma}
	For any $g∈G$, the set $T = §{ g^i : 0⩽ i< \tilenum(g) }$ tiles $G$.
\end{lemma}
\begin{proof}
	Denote by $H$ the subgroup generated by $g$ and consider the action of $H$ on $G$ by multiplication on the left. Then define a set $C$ by choosing one element from each orbit, so that $§{Hc : c ∈ C}$ is a partition of $G$. If $g$ has finite order, then $H = T$ and the assertion is proved.	
	If instead $g$ has infinite order, then $H$ is the disjoint union $\bigcup_{k∈\zz} T g^{2k}$ and thus $§{Tg^{2k}c : k∈\zz, c ∈ C}$ is a partition  of $G$, witnessing that $T$ tiles $G$.
\end{proof}

\begin{theo}[characterisation of functional freeness]\label{characterisation of functional freeness}
	Let $G$ be a countable amenable group. 
	A $G$-$L_p$ lattice $E$ is functionally free if and only if for each $g∈G\setdiff§{1_G}$, each $u∈E_+$, and any $ε>0$, there exist $v$ and $w$ in $E_+$ such that
	\begin{enumerate}
		\item $v ⩽ u + w$ and $\norm{w} ⩽ ε$,
		\item $\norm{v \meet g v} ⩽ ε$,
		\item $\norm{v} ⩾ \frac 13 \norm{u} - ε$.
	\end{enumerate}
\end{theo}
\begin{proof}
	Suppose first that the action is functionally free and fix $g∈G\setdiff§{1_G}$, $u∈E_+$, and $ε>0$. By the previous lemma, we know that $T = §{g^i : 0⩽ i < \tilenum(g)}$ tiles $G$, so we can apply \Cref{rokh lemma} to $u$, with $E=F$ equal to the $\gen{g}$-invariant band generated by $u$, with $ε$ replaced by $(ε/\norm{2u})^p$, and $£T = §{T}$, $K=§{1_G}$, $δ=0$. This yields some positive $u_{\mathrm{err}}$ and $u_T$ such that $§{ u_{\mathrm{err}} } ∪ §{ t u_T : t∈T }$ is a maximal disjoint subset of $F$, and $w=u \restr  u_{\mathrm{err}}$ has norm at most $ε$.
	
	Define now
	\begin{equation*}
		u_{\mathrm{even}} = \sum_{k<\floor{\frac{\tilenum(g)}2}} g^{2k} u_T,
		\quad
		u_{\mathrm{odd}} = \sum_{k<\floor{\frac{\tilenum(g)}2}} g^{2k+1} u_T,
		\quad
		u_{\mathrm{rem}} = \sum_{k<\tilenum(g)} g^{k} u_T - u_{\mathrm{even}} - u_{\mathrm{odd}},
	\end{equation*}
	so that, together with $u_{\mathrm{err}}$, they still form a maximal disjoint subset of $F$, and thus
	\begin{equation*}
		u \restr u_{\mathrm{even}} + u\restr u_{\mathrm{odd}} + u \restr u_{\mathrm{rem}} = u - w.
	\end{equation*}
	One of the three addends on the left-hand side above must have norm at least $\frac13 \norm{u-w} = \frac 13 \norm{u} - \frac{ε}{3}$. We may denote this element by $v$, and notice that it satisfies the three conditions above.
	
	We shall now deal with the other direction. Let then $g∈G\setdiff§{1_G}$ and  $u∈E_+$. Set $n=\tilenum(g)$ and choose $ε< \norm{u} / 12$. Suppose now that there exist $v$ and $w$ as in the statement of the present lemma. We set
	\begin{equation*}
		v' = (v-w)^+ = v-(v\meet w),
	\end{equation*}
	so that $0 ⩽ v' ⩽ u$ and 
	\begin{equation*}
		\norm{v'} ⩾ \norm{v} - \norm{w} ⩾ \frac 13 \norm{u} - 2ε.
	\end{equation*}
	Define now 
	\begin{equation*}
		v'' = v' - §( (v'\meet gv') \join (v' \meet g^{-1}v') ),
	\end{equation*}
	and notice that $v''$ is still in the band generated by $u$, but it is also disjoint from $g v''$. In order to conclude, it remains to show that $v'' ≠ 0$.
	By the second condition in the statement of the lemma, we have  $\norm{v'\meet g^{-1}v'} = \norm{v'\meet gv'} ⩽ \norm{v \meet gv } ⩽ ε$, which implies 
	\begin{equation*}
		\norm{v''} ⩾ \norm{v'} - 2ε ⩾ \frac13 \norm{u} - 4ε > 0,
	\end{equation*}
	by the choice of $ε$.
\end{proof}

\subsection{Model theory of $G$-$L_p$ lattices}
The reader unfamiliar with continuous model theory is referred to \cite{cont-logic}.
We denote by $£L_{\mathrm{BL}}$ the continuous language of Banach lattices, as defined in \cite{canonicalBases} or in \cite{lpIndep}. Given a countable group $G$, we expand $£L_{\mathrm{BL}}$ with unitary functions $\dot{g}$ for each element $g∈G$, and we denote this new language by $£L_G$.

Recall that $\alpl$ is the theory of atomless $L_p$ lattices in $£L_{\mathrm{BL}}$. Consider now the theory $T_G$ obtained by adding to $\alpl$  the axioms saying that $G$ acts on a Banach lattice by automorphisms. More precisely, for each $g∈G$ we add an axiom stating that $\dot g$ is an automorphism of normed vector lattices, and for each $g,h,k ∈G$ satisfying $g = hk$, we add
\begin{equation*}
	\sup_{x} \norm{ \dot g (x) - \dot h §(\dot k (x))} =0.
\end{equation*}
Notice that if $ρ$ is an action of $G$ on a Banach lattice $E$, then $ρ(g)$ is the interpretation of $\dot g$ in the $£L_G$-structure $(E,ρ)$. For simplicity, we will henceforth use the same symbol $g$ for the element of $G$, the automorphism $ρ(g)$, and the function symbol $\dot g$.

Using the characterisation in \Cref{characterisation of functional freeness}, the notion of functional freeness can be expressed in $£L_G$ via the sentences
\begin{equation}\label{freeness axiom}
	\adjustlimits \sup_{x} \inf_{y} 
	\max§(
	\norm{(\abs{x}-\abs{y})^-},\;
	\norm{ g\abs{y} \meet \abs{y} }[big],\;
	\frac13 \norm{x} - \norm{y}
	)[bigg]
	=0
	,
	\tag{Freeness$_g$}
\end{equation}
for each non-identical $g ∈ G$. We denote by $\tgfree$ the theory obtained by adding these axioms to $T_G$.

\begin{theo}[model completeness]\label{model completeness of tgfree}
	The theory $\tgfree$ is model complete.
\end{theo}

To prove this result, we will follow the same basic lines as in the proof of \cite[Lemma~3.5]{scielzo}, albeit with the added steps required to deal with the fact that here we have multiple towers and we need to keep their boundaries small. 

\begin{proof}
	Take two $\aleph_1$-saturated models $M⊆ N$ of $\tgfree$, a quantifier-free basic formula $φ(x,y)$, some positive elements $w∈M$ and $v∈N$, and a positive number $ε$.
	It suffices to find $\hat v ∈ M$ such that $\abs{ φ(\hat v,w) - φ(v,w) } < ε$.
	As $φ$ is quantifier-free, we can rewrite it as
	\begin{equation*}
		φ(x,w) = ψ(g_1 x, …,g_m x ; w_1,…,w_\ell ) ,
	\end{equation*}
	for some elements $g_1,…,g_m ∈ G$ and some parameters $w_1,…,w_\ell ∈ M$, with no other occurrence of elements of $G$ in $ψ$. We can harmlessly suppose that $w = (w_1,…,w_\ell)$ and that the set $K = §{g_1, …,g_m}$ is symmetric and contains the identity (if not, just replace it by $KK^{-1}$). We will also use the name $K$ for both the set and the tuple of the $g_i$'s, so in particular the equality above can be written more concisely as $φ(x,w) = ψ(Kx, w)$. 
	
	As $ψ$ is a continuous combination of norms of terms $\norm{r(Kx,w)}$, there is some $δ_0>0$ such that, for all $\hat v ∈ M$, the inequality $\abs{ φ(v,w) - φ(\tilde v,w) } < ε$ holds if 
	\begin{equation*}
		\abs{
			\norm{ r(K v, w)} - \norm{ r(K \hat v, w)}
		}[Big]
		⩽ δ_0 
	\end{equation*}
	for all terms $r$ in $ψ$. This means that we just need to find  $\hat v ∈ M$ satisfying the inequality above.
	
	Consider now the smallest $G$-invariant band in $M$ containing $w$---let us call it $B$. This is a principal band, as it is generated for instance by
	\begin{equation*}
		w_0 = \sum_{g∈G} \frac{1}{2^{ι(g)}} gw,
	\end{equation*}
	where $ι$ is any injection of $G$ into $\zz_{>0}$. This means that $w_0$ is a weak unit of $B$, which is an $L_p$ lattice in its own right, acted upon by $G$.
	We shall now decompose $v$ into two disjoint components: $v' =  v \restr w_0$ in $C=\band[N]{w_0}$, i.e., the band of $w_0$ in $N$, and $v'' = v-v'$ in $C^\perp$, the disjoint complement of $w_0$ relative to $N$. Notice that $C$ is a $G$-invariant sublattice of $N$, and contains $B$ as a sublattice.
	
	As $\norm{r(\;\cdot\;, w)}{}^N$ is uniformly continuous, there exists some $δ_1>0$ such that, for all $a_1,a_2∈N^{\card{K}}$, 
	\begin{equation}\label{proof: unif continuity}
		d_\infty(a_1,a_2) <δ_1 \implies \abs{\norm{r(a_1,w)} - \norm{r(a_2,w)}}[Big] < \frac{δ_0}{4},
	\end{equation}
	where $d_\infty$ is defined by $d_\infty§((x_i)_i,(y_i)_i) = \max_i \norm{x_i-y_i}$.
	Consider now
	\begin{equation*}
		ε_0 = §( \frac{δ_1}{2\norm{w_0+v'}+4\norm{v''}} )^p
	\end{equation*}
	and let $δ$ be as small as in \Cref{amenable rokhlin lemma}, with $ε$ replaced by $ε_0$. 
	By \Cref{exact tiling}, there is a tiling family $£T$ of $G$ whose shapes are all $(K,δ)$-invariant, so that we can apply \Cref{rokh lemma} to $v'$, with $F= B$ and $E= C$.
	We can then find positive elements $u'_{\mathrm{err}}, u'_{T}∈F ⊆ M$ for each $T∈£T$,  such that
	\begin{enumerate}[label=(\alph*)]
		\item \label{proof: eps lemma 1}$§{ u'_{\mathrm{err}} } ∪ §{ t u'_T : t∈T∈£T }$ is a maximal disjoint subset of $F$,
		\item \label{proof: eps lemma} $\displaystyle \sum_{T∈£T}\, \sum_{s∈\inboundary{K} T} \norm{ v'\restr s u'_T} + \norm{ v' \restr  u'_{\mathrm{err}} } ⩽ 2ε_0^{1/p} \norm{w_0+v'} ⩽  δ_1$.
	\end{enumerate}
	Notice that since the elements $u'_{\scriptscriptstyle (\,⋅\,)}$ are in $M$, so are the restrictions of the parameters $w\restr u'_{\scriptscriptstyle (\,⋅\,)}$. We can do the same for $v''$ in place of $v'$, but with $w_0 = \sum_{g∈G} \frac{1}{2^{ι(g)}} v''$ and $F=E= C^{\perp}$, and find elements $u''_{\scriptscriptstyle (\,⋅\,)}$ satisfying \ref{proof: eps lemma 1} and \ref{proof: eps lemma} above, with the single primes replaced by double primes.

	For each $T∈£T$ and $s∈T$, we define now
	\begin{equation*}
		v'_{T,s} = v'\restr s u'_T \quad \text{ and }\quad  v'_{\mathrm{err}} = v'\restr u'_{\mathrm{err}},
	\end{equation*}
	and similarly for the doubly primed elements. 
	Observe that these restrictions $v'_{\mathrm{err}}$, $v''_{\mathrm{err}}$, $v'_{T,s}$, and $v''_{T,s}$ with varying $s∈T∈£T$ form a disjoint decomposition of $v$ in $N$. 
	
	We will deal with $v'$ and $v''$ separately, starting with $v'$.
	Consider the Banach-lattice type of $(s^{-1} v'_{T,s} : s ∈ T∈£T)$ over $§{ s^{-1} w \restr u'_T : s ∈ T∈ £T }* ∪ §{ w\restr u'_{\mathrm{err}} }* ∪ §{ u'_T : T∈£T }*$ in the reduct $N \restr \lBL$. 
	By model completeness of $\alpl$, this is a type in $M \restr \lBL$, and by saturation, it is realised by some tuple $(s^{-1}\hat v'_{T,s} : s ∈ T∈£T)$ in $B$, with each $s^{-1}\hat v'_{T,s}$ supported by the corresponding $u'_T$. This means in particular that the elements $\hat v'_{T,s}$, with varying $s ∈ T$ and $T∈ £T$, are pairwise disjoint. 
	
	For any $g∈K$, define 
	\begin{equation*}
		v'_g = \sum_{T∈£T} \sum_{\substack{s∈ T \\ g^{-1}s ∈ T}} v'_{T,g^{-1}s},
	\end{equation*}
	and $\hat v'_g$ in the same way. 
	As 
	\begin{math}
		v'_g = \sum_{T∈£T} \sum_{\substack{s∈ T ∩ g^{-1}T}} v'_{T,s}
	\end{math}
	and $T\setdiff g^{-1}T ⊆ \inboundary{K} T$, 
	we have
	\begin{align*}
		\norm{gv'-gv'_g} = \norm{v' - v'_g} 
		&= \sum_{T∈£T} \sum_{s∈ T \setdiff g^{-1}T} \norm{v'_{T,s}} + \norm{v'_{\mathrm{err}}}\\
		&⩽ \sum_{T∈£T} \sum_{ s∈ \inboundary{K} T } \norm{v'_{T,s}} + \norm{v'_{\mathrm{err}}} ⩽ δ_1,
	\end{align*} 
	by point \ref{proof: eps lemma} above. It then follows from \eqref{proof: unif continuity} that 
	\begin{equation*}
		\abs{ \norm{r(Kv',w)} - \norm{r(g v'_g:g∈K,w)} }[Big]
		⩽ \frac{δ_0}{4}
	\end{equation*}
	and the same is true if we replace $v'$ with $\hat v'$. We shall now prove that
	\begin{equation}\label{proof: eq 3}
		\norm{r(g v'_g:g∈K,w)} = \norm{r(g \hat v'_g:g∈K,w)},
	\end{equation}
	from which it will follow
	\begin{equation}\label{proof: dis v}
		\abs{ \norm{r(Kv',w)} - \norm{r(K\hat v',w)} } 
		⩽ \frac{δ_0}{2}.
	\end{equation}

	For convenience, if $h∉T$, we set $v'_{T,h} = 0$. Notice that $g v'_{T,g^{-1}s}$ is supported by $s u'_{T}$ (trivially when $g^{-1}s ∉T$).
	To compare the terms in the $v_g$'s and thus prove \eqref{proof: eq 3}, we exploit the disjoint decomposition given by the $su'_T$'s and $u'_{\mathrm{err}}$.  Since the restriction commutes with every other vector lattice operation, this yields
	\begin{align*}
		\norm{r(g v'_g:g∈K,w)} 
		%&= \sum_{T∈£T} \sum_{s∈ T} \norm{r(g v'_g : g∈K, w ) \restr s u'_{T} } + \norm{r(g v'_g : g∈K, w ) \restr u'_{\mathrm{err}} }\\
		&= \sum_{T∈£T} \sum_{s∈ T} \norm{r(g v'_g \restr s u'_{T} : g∈K, w  \restr s u'_{T} ) } \\[-1ex]
		&\hspace{5em}+ \norm{r(g v'_g\restr u'_{\mathrm{err}} : g∈K, w \restr u'_{\mathrm{err}} )  } \\
		\intertext{Observe that $g v'_g \restr s u'_{T} = g (v'_g \restr g^{-1}s u'_{T})$, which is precisely $gv'_{T,g^{-1}s}$ (even if $g^{-1}s ∉ T$), and no $v'_g$ has any component supported by $u'_{\mathrm{err}}$, so the sum above is equal to}
		&\mathrel{\phantom{=}} \sum_{T∈£T} \sum_{s∈ T} \norm{r(g v'_{T,g^{-1}s} : g∈K, w  \restr s u'_{T} ) } + \norm{r(\vec{0}, w \restr u'_{\mathrm{err}} )  },\\
		\intertext{and by applying the corresponding $s^{-1}∈T^{-1}$ to each term, this can be rewritten as}
		&\mathrel{\phantom{=}} \sum_{T∈£T} \sum_{ s∈ T} \norm{r(s^{-1}g v'_{T,g^{-1}s}  : g∈K, s^{-1}w \restr u'_{T}) }   
		+ \norm{r(\vec{0}, w \restr u'_{\mathrm{err}} )  }. \\
		\intertext{In this sum, the first argument of $r$ is always a tuple of zeroes and elements of the form $t^{-1} v'_{T,t}$ with $t∈T ∈ £T$, hence by type equivalence, it is equal to }
		&\mathrel{\phantom{=}} \sum_{T∈£T} \sum_{ s∈ T } \norm{r(s^{-1}g\hat v'_{T,g^{-1}s}  : g∈K , s^{-1}w \restr u'_{T}) }
		+ \norm{r(\vec{0}, w \restr u'_{\mathrm{err}} )  }.
	\end{align*}
	Repeating the first three steps with $v'_{(\,⋅\,)}$ replaced by $\hat v'_{(\,⋅\,)}$, we get that the last sum is equal to $\norm{r(g \hat v'_g:g∈K,w)}$, and this proves \eqref{proof: eq 3}.
	
	We shall now deal with the external part $v''$ of $v$. We shall firstly clear out some space in $M$ for a realisation of $v''$ by finding an element $w_0' ∈ M$ generating a $G$-invariant band disjoint from $w_0$, and thus disjoint from $C$ when regarded as an element of $N$.	This is possible because $M$ is $\aleph_1$-saturated. 
	
	Then, we apply \Cref{rokh lemma} to $w_0'$ with $F=E=\band[M]{w_0'}$, and obtain a decomposition of $w_0'$ given by elements $\tilde u_T ∈ M$, for each $T ∈ £T$, and $\tilde u_{\mathrm{err}} ∈ M$.
	At this point, for each $T ∈ £T$, we find a tuple $(s^{-1}\hat v''_{T,s} :s ∈T)$ in $\band[M]{\tilde u_T}$ realising the Banach-lattice type of $(s^{-1}v''_{T,s} :s ∈T)$ over the empty set. This is possible because every atomless $L_p$ lattice contains a separable atomless $L_p$ sublattice, which is separably categorical and thus $1$-saturated. 
	By the choice of the $\tilde u_T$'s, we have that the $\hat v''_{T,s}$'s with varying $s∈T∈£T$ are pairwise disjoint.
	We can then proceed as we did before with $v'$ and $\hat v'$ to deduce 
	\begin{equation*}\label{}
		\abs{ \norm{r(Kv'',w)} - \norm{r(K\hat v'',w)} }[Big]
		⩽ \frac{δ_0}{2}.
	\end{equation*}
	This inequality, together with \eqref{proof: dis v}, completes the proof.
\end{proof}

In order to prove that $\tgfree$ is the model companion of $T_G$, we will need the existence of products of arbitrary semi-finite measure spaces. 
\begin{fact}[{\cite[251I and 251J]{fremlin2}}]\label{product measure space}
	Let $(X,£F,μ)$ and $(Y,£G,ν)$ be semi-finite measure spaces, and let $£F \otimes £G$ be the \tsigma-algebra generated by $§{A×B : A∈£F, B∈£G}$. There exists a semi-finite measure $λ$ on $(X×Y,£F\otimes £G)$ satisfying
	\begin{math}
		λ(A×B) = μ(A)ν(B),
	\end{math}
	for all $A∈£F$ and $B∈£G$.
\end{fact}

\begin{theo}
	$\tgfree$ is the model companion of $T_G$.
\end{theo}
\begin{proof}
	We have just proved that $\tgfree$ is model complete, so we just need to show that every model of $T_G$ embeds in a model of $\tgfree$, the converse being trivial. 
	
	Consider a concrete $G$-$L_p$ lattice $ E = L^{p}(Z,Σ,μ,π)$ and endow $(G,£P(G))$ with a probability measure $ν$ which is positive on every singleton (for instance, if $ι$ is an injection of $G$ into $\zz_{>0}$, then $ν§{g} = 1/2^{ι(g)}$ defines a suitable measure). This way the action $λ$ by multiplication on the left  is non singular with respect to $ν$.
	
	Let $£X$ be a product of $(Z,Σ,μ)$ and $(G,£P(G),ν)$, in the sense of \Cref{product measure space}, and let $π\otimes λ$ the map defined by
	\begin{equation*}
		(π\otimes λ) (g) (z,h) = §(π(g)(z) , gh),
	\end{equation*}
	for all $g,h ∈ G$ and $z∈Z$. One can check that $π\otimes λ$ is indeed a non-singular action on $£X$ and that it is also free, since for any non-identical $g∈G$ and set $A ∈ £X$, there is some $h∈G$ such that $B = A ∩ (Z × §{h})$ has positive measure, and $gB$ is disjoint from $B$. Since $£X$ is atomless, this means that the $G$-$L_p$ lattice $ F = L^{p}(£X,π\otimes λ)$ is a model of $T_G^{\mathrm{free}}$. 
	Finally, $E$ can be embedded into $F$ via the map $Φ$ defined by $Φ(v)(z,g) = v(z)$. 
\end{proof}

We will show that $\tgfree$ has quantifier elimination following the same lines as in \cite{scielzo}, with the additional crucial observation that the algebraic closure of a subset of an $L_p$ lattice coincides with the sublattice generated by it \cite[Lemma~3.12]{lpIndep}.

Given a cardinal $κ>2^{\aleph_0}$, recall that a structure is said to be \emph{$κ$-universal} if it is $κ$-strongly homogeneous and $κ$-saturated. Henceforth, we will work in a $κ$-universal model $(£U,ρ)$ of $T_G$, for some large $κ$.

\begin{theo}\label{qe of TA}
	$\tgfree$ has quantifier elimination.
\end{theo}
\begin{proof}
	Since $\tgfree$ is the model companion of $T_G$, we just need to prove that $(T_G)_{∀}$ has the amalgamation property, that is, any two models $(E_1,ρ_1)$ and $(E_2,ρ_2)$ of $T_G$ can be amalgamated over a common $£L_G$-substructure $A$. We may assume that the $£L_{\mathrm{BL}}$-structures $E_1$ and $E_2$ are in $£U$, and that $E_1$ is independent of $E_2$ over $A$, so that $E_1 ∩ E_2 = \acl A$, which is precisely the Banach lattice generated by $A$.
	This means that if $g∈G$, then $ρ_1(g)$ and $ρ_2(g)$ coincide on $\acl A$ too, so we may as well assume that $A$ is algebraically closed, in the sense of $£L_{\mathrm{BL}}$-structures.
	
	Since \alpl\ is stable, the automorphisms of its models can be amalgamated (see \cite[Thereom~3.3]{lascar}), meaning that if $α_1 ∈ \aut(E_1)$ and $α_2 ∈ \aut(E_2)$ coincide over $A$, then $α_1 ∪ α_2$ is elementary on $E_1 ∪ E_2$, and it may thus be extended to a unique automorphism  $α$ of $E_3 = \acl(E_1 ∪ E_2)$. It follows that there is a $ρ_3\colon G \to \aut(E_3)$ satisfying $ρ_3(g) \restr E_1 = ρ_1(g)$ and $ρ_3(g) \restr E_2 = ρ_2(g)$ for all $g∈G$. 
	Moreover, for any $g,h∈G$, $ρ_3(g)ρ_3(h)$ coincides with $ρ_1(g)ρ_1(h) ∪ ρ_2(g)ρ_2(h) = ρ_1(gh) ∪ ρ_2(gh)$ on $E_1 ∪ E_2$, which coincides with $ρ_3(gh)$, so by uniqueness of the extension, $ρ_3(gh) = ρ_3(g)ρ_3(h)$, showing that $(E_3,ρ_3)$ is a model of $T_G$ amalgamating $(E_1,ρ_1)$ and $(E_2,ρ_2)$ over $A$.
\end{proof}

An immediate corollary of quantifier elimination is that the theory $\tgfree$ is complete, since the only constant is $0$, which is fixed by all functions of $£L_G$. Suppose now that $(£U, ρ)$ is a sufficiently large universal model of $\tgfree$, and denote the orbit of $a∈£U$ by $Ga$. Another corollary of quantifier elimination is the following.

\begin{lemma}
	Let $\seq{a}$ and $\seq{b}$ be two tuples of the same lengths in $£U$, and $C$ a small subset of $£U$. Then $\seq{a}$ and $\seq{b}$ have the same type over $C$ in the sense of $(£U,ρ)$ if and only if $G\seq{a}$ and $G\seq{b}$ have the same type over $GC$ in the sense of $£U$.
\end{lemma}

As in the case where $G = \zz$, we can define the ternary relation $A \indep{\ensuremath{G}}_C B$ to mean that $\dcl§(G(AC))$ is forking-independent of $\dcl§(G(BC))$ over $\dcl§(G(C))$, which is equivalent to $GA$ being forking-independent of $GB$ over $GC$, because the definable closure of a set is just the Banach lattice generated by it. 

With the same proof of \cite[Lemma~3.13]{scielzo}, \textit{mutatis mutandis}, one can show that this relation $\indep{\ensuremath{G}}$ satisfies symmetry, transitivity, finite character, extension, local character, and stationarity, thus proving the following. 

\begin{theo}
	$\tgfree$ is stable and the relation $\indep{\ensuremath{G}}$ coincides with forking independence.
\end{theo}

The proof of \Cref{amenable rokhlin lemma} crucially relies on $G$ being amenable, and we do know if a decomposition theorem is still possible for non-amenable groups, not even for free group on two generators $F_2$. Since this decomposition was our main tool for proving the model completeness of $\tgfree$, this raises the following question.

\begin{question}
	When $G$ is not amenable, does $T_G$ admit a model companion?
\end{question}

A result in this direction is given in \cite{free-groups-actions}, where the authors prove that the theory of probability measure-preserving actions of free groups has a model companion. So far, we do not know if this result can be generalised to the case of actions of free groups on $L_p$ lattices.

\nocite{hodges}

\printbibliography

\end{document}